\documentclass[reqno]{amsart}
\usepackage{amssymb}
\usepackage{amsmath}
\usepackage{amsfonts}

\newtheorem{theorem}{Theorem}[section]
\theoremstyle{plain}

\newtheorem{lemma}{Lemma}[section]

\newtheorem{remark}{Remark}[section]

\numberwithin{equation}{section}

\begin{document}
\title[Exponential decay of solutions]{Exponential decay of solutions for
the plate equation with localized damping}
\author{ Sema Simsek\ \ }
\address{{\small Department of Mathematics,} {\small Faculty of Science,
Hacettepe University, Beytepe 06800}, {\small Ankara, Turkey}}
\email{semasimsek@hacettepe.edu.tr}
\author{ \ Azer Khanmamedov}
\email{azer@hacettepe.edu.tr}

\begin{abstract}
In this paper, we give positive answer to the open question raised in [E.
Zuazua, Exponential decay for the semilinear wave equation with localized
damping in unbounded domains. J. Math. Pures Appl., 70 (1991) 513--529] on
the exponential decay of solutions for the\ semilinear plate equation with
localized damping.
\end{abstract}

\maketitle

\section{Introduction}

In this paper, we consider the exponential decay of solutions for the plate
equation%
\begin{equation}
u_{tt}+\Delta ^{2}u+a(x)u_{t}+\alpha u+f(u)=0\text{, \ \ }(t,x)\in (0,\infty
)\times
\mathbb{R}
^{n}\text{,}  \tag{1.1}
\end{equation}%
with the initial conditions%
\begin{equation}
u(0,x)=u_{0}(x)\text{, \ }u_{t}(0,x)=u_{1}(x)\text{, \ \ }x\in
\mathbb{R}
^{n}\text{,}  \tag{1.2}
\end{equation}%
where $\alpha >0,$ and the functions $a\left( \cdot \right) $, $f\left(
\cdot \right) $ satisfy the following conditions%
\begin{equation}
a\in L^{\infty }(%
\mathbb{R}
^{n})\text{, }a(\cdot )\geq 0, \text{ a.e. in }%
\mathbb{R}
^{n}\text{,}  \tag{1.3}
\end{equation}%
\begin{equation}
a(\cdot )\geq a_{0}>0\text{ a.e. in }\left\{ x\in
\mathbb{R}
^{n}:\left\vert x\right\vert \geq r_{0}\right\} \text{, for some }r_{0}\text{%
,}  \tag{1.4}
\end{equation}%
\begin{equation}
f\in C^{1}(%
\mathbb{R}
)\text{, }\left\vert f^{\prime }(s)\right\vert \leq C\left( 1+\left\vert
s\right\vert ^{p-1}\right) \text{, }p>1\text{, }(n-4)p\leq n\text{,}
\tag{1.5}
\end{equation}%
\begin{equation}
\text{ }f(s)s\geq 0\text{, for every }s\in
\mathbb{R}
\text{.}  \tag{1.6}
\end{equation}%
\ \

By the semigroup theory, it is well known that under conditions (1.3), (1.5)
and (1.6), for every $\left( u_{0},u_{1}\right) \in H^{2}\left(
\mathbb{R}
^{n}\right) \times L^{2}\left(
\mathbb{R}
^{n}\right) $, problem (1.1)-(1.2) has a unique weak solution in $C\left(
[0,\infty );H^{2}\left(
\mathbb{R}
^{n}\right) \right) \cap $ $C^{1}\left( [0,\infty );L^{2}\left(
\mathbb{R}
^{n}\right) \right) $. The energy functional of problem (1.1)-(1.2) is%
\begin{equation*}
E\left( t,u_{0},u_{1}\right) =\frac{1}{2}\int\nolimits_{%
\mathbb{R}
^{n}}\left( \left\vert u_{t}\left( t,x\right) \right\vert ^{2}+\left\vert
{\small \Delta u}\left( t,x\right) \right\vert ^{2}+\alpha \left\vert
{\small u}\left( t,x\right) \right\vert ^{2}\right) dx+\int\nolimits_{%
\mathbb{R}
^{n}}F\left( u(t,x)\right) dx\text{,}
\end{equation*}%
where $u\left( t,x\right) $ is a weak solution of (1.1)-(1.2) with initial
data $\left( u_{0},u_{1}\right) \in H^{2}\left(
\mathbb{R}
^{n}\right) \times L^{2}\left(
\mathbb{R}
^{n}\right) $ and $F\left( z\right) =\int\nolimits_{0}^{z}f\left( s\right)
ds $, for all $z\in
\mathbb{R}
$.

Exponential decay of the energy for problem (1.1)-(1.2) means that there
exist some constants $C>1$, $\gamma >0$ such that%
\begin{equation*}
E\left( t,u_{0},u_{1}\right) \leq CE\left( 0\,,u_{0},u_{1}\right) e^{-\gamma
t\text{ }}\text{, \ }\forall t\geq 0\text{,}
\end{equation*}%
for every $\left( u_{0},u_{1}\right) \in H^{2}\left(
\mathbb{R}
^{n}\right) \times L^{2}\left(
\mathbb{R}
^{n}\right) $.

Energy decay of the solutions for wave and plate equations has been studied
by many authors under different conditions. We refer to [1-7] for wave
equations and [8-13] for plate equations. In \cite{2} and \cite{3}, the
author showed that the semilinear wave equation with localized damping has
an exponential energy decay under suitable conditions in bounded and
unbounded domains, by reducing the question to a unique continuation problem
which was solved by applying results of \cite{14}. However, the exponential
decay of the energy for (1.1)-(1.2) was introduced as an open question in
\cite[Remark 3.2]{3}, since the techniques of that article were not enough
to obtain the desired result. This is caused by the lack of unique
continuation result for the weak solutions of the plate equation with
nonsmooth coefficients.

The main goal of this paper is to answer this open question. To this end,
using the sequentially limit transition technique (see [15-17]), we firstly
prove the uniformly asymptotic compactness of the family of semigroups (see
Lemma 2.3). Then, using point dissipativity property for the semilinear
plate equation established in \cite{18} and borrowing the energy
inequalities obtained in \cite{3} in the superlinear case, we show the
contraction of the energy for the plate equations (see Lemma 2.6), which
leads to exponential decay of energy for problem (1.1)-(1.2).

Our main result is as follows:

\begin{theorem}
Assume conditions (1.3)-(1.6) hold. Additionally, suppose that either\newline
$\left( i\right) $ (The globally Lipschitz case). $f^{\prime }\in L^{\infty
}\left(
\mathbb{R}
\right) $ and%
\begin{equation}
\underset{s\rightarrow -\infty }{\lim }\frac{f(s)}{s}=\alpha _{1}\in \lbrack
0,\infty )\text{, }\underset{s\rightarrow \infty }{\lim }\frac{f(s)}{s}%
=\alpha _{2}\in \lbrack 0,\infty )\text{,}  \tag{1.7}
\end{equation}%
or\newline
$\left( ii\right) $ (The superlinear case). There exists some $\delta >0$
such that%
\begin{equation}
f\left( s\right) s\geq \left( 2+\delta \right) F\left( s\right) \text{, }%
\forall s\in
\mathbb{R}
\text{.}  \tag{1.8}
\end{equation}%
Then there exist some constants $C>1$ and $\gamma >0$ such that the estimate
\begin{equation*}
E\left( t,u_{0},u_{1}\right) \leq CE\left( 0\,,u_{0},u_{1}\right) e^{-\gamma
t\text{ }}\text{, \ }\forall t\geq 0\text{,}
\end{equation*}%
holds for every weak solution $u\left( t,x\right) $ of (1.1)-(1.2) with
initial data $\left( u_{0},u_{1}\right) \in H^{2}\left(
\mathbb{R}
^{n}\right) \times L^{2}\left(
\mathbb{R}
^{n}\right) $.
\end{theorem}

\begin{remark}
We note that applying the method of this paper, of course, using suitable
multipliers for the proof of Lemma 2.3, one can prove the exponential decay
of the weak solutions for the initial boundary value problem%
\begin{equation*}
\left\{
\begin{array}{c}
u_{tt}+\Delta ^{2}u+a(x)u_{t}+\alpha u+f(u)=0\text{, \ \ }(t,x)\in (0,\infty
)\times \Omega \text{,} \\
u(t,x)=\frac{\partial }{\partial \nu }u(t,x)=0\text{, \ \ \ \ \ \ \ \ \ \ \
\ \ \ \ \ \ }(t,x)\in (0,\infty )\times \partial \Omega \text{,} \\
u(0,x)=u_{0}(x)\text{, \ \ \ \ }u_{t}(0,x)=u_{1}(x)\text{, \ \ \ \ \ \ \ \ \
\ \ \ \ \ \ \ \ \ }x\in \Omega \text{,}%
\end{array}%
\right.
\end{equation*}%
where $\alpha >0$, $\Omega \subset
\mathbb{R}
^{n}$ is a domain with smooth boundary, $\nu $ is outer unit normal vector,
the nonlinear function $f(\cdot )$ satisfies the conditions (1.5), (1.6) and
either (1.7) or (1.8), the damping coefficient $a(\cdot )$ satisfies the
following conditions%
\begin{equation*}
a\in L^{\infty }(\Omega )\text{, }a(\cdot )\geq 0,\text{ a.e. in }\Omega
\text{,}
\end{equation*}%
and
\begin{equation*}
a(\cdot )\geq a_{0}>0\text{ a.e. in }\omega \text{, for some }\omega \subset
\Omega \text{, such that }
\end{equation*}%
\begin{equation*}
\omega =\left\{
\begin{array}{c}
\text{ a neighbourhood of the boundary }\partial \Omega \text{, \ \ \ \ if }%
\Omega \text{ is bounded, \ \ \ \ \ \ \ \ \ \ \ \ \ \ \ \ \ \ \ \ \ \ \ \ \
\ \ \ \ \ \ \ \ \ } \\
\text{the union of a neighbourhood of the boundary }\partial \Omega \text{
and }\left\{ x\in \Omega :\left\vert x\right\vert \geq r_{0}\right\} \text{,
if }\Omega \text{ is unbounded.}%
\end{array}%
\right.
\end{equation*}
\end{remark}

\bigskip

\section{Proof of Theorem 1.1}

We start with the following lemmas.

\begin{lemma}
Let us assume the condition (1.5) is satisfied . If the sequence $\left\{
u_{k}\right\} _{k=1}^{\infty }$ weakly converges in $H^{2}(%
\mathbb{R}
^{n})$ and the positive sequence $\left\{ \lambda _{k}\right\}
_{k=1}^{\infty }$ converges, then there exists $C=C(f,\underset{k}{\sup }%
\lambda _{k},$ $\underset{k}{\sup }\left\Vert u_{k}\right\Vert _{H^{2}(%
\mathbb{R}
^{n})})>0$ such that%
\begin{equation*}
\underset{m\rightarrow \infty }{\lim \sup }\underset{k\rightarrow \infty }{%
\lim \sup }\left\Vert \frac{1}{\lambda _{k}}f(\lambda _{k}u_{k})-\frac{1}{%
\lambda _{m}}f(\lambda _{m}u_{m})\right\Vert _{L^{2}(%
\mathbb{R}
^{n})}
\end{equation*}%
\begin{equation}
\leq C\underset{m\rightarrow \infty }{\lim \sup }\underset{k\rightarrow
\infty }{\lim \sup }\left\Vert u_{k}-u_{m}\right\Vert _{H^{2}(%
\mathbb{R}
^{n})}\text{.}  \tag{2.1}
\end{equation}%
Furthermore, if, additionally, condition (1.7) is satisfied, then (2.1) also
holds for $\lambda _{k}\rightarrow \infty $, with the constant $C$ depending
only on $f$.
\end{lemma}

\begin{proof}
Let $u_{k}\rightarrow u$ weakly in $H^{2}(%
\mathbb{R}
^{n})$ and $\lambda _{k}\rightarrow \lambda _{0}\in \left[ 0,\infty \right] $%
. By triangle inequality, we have%
\begin{equation*}
\left\Vert \frac{1}{\lambda _{k}}f(\lambda _{k}u_{k})-\frac{1}{\lambda _{m}}%
f(\lambda _{m}u_{m})\right\Vert _{L^{2}(%
\mathbb{R}
^{n})}\leq \left\Vert \frac{1}{\lambda _{k}}f(\lambda _{k}u_{k})-\frac{1}{%
\lambda _{k}}f(\lambda _{k}u)\right\Vert _{L^{2}(%
\mathbb{R}
^{n})}
\end{equation*}%
\begin{equation}
+\left\Vert \frac{1}{\lambda _{k}}f(\lambda _{k}u)-\frac{1}{\lambda _{m}}%
f(\lambda _{m}u)\right\Vert _{L^{2}(%
\mathbb{R}
^{n})}+\left\Vert \frac{1}{\lambda _{m}}f(\lambda _{m}u)-\frac{1}{\lambda
_{m}}f(\lambda _{m}u_{m})\right\Vert _{L^{2}(%
\mathbb{R}
^{n})}\text{.}  \tag{2.2}
\end{equation}%
Since $H^{2}(%
\mathbb{R}
^{n})\subset L^{\frac{2n}{(n-4)^{+}}}(%
\mathbb{R}
^{n})\cap L^{2}(%
\mathbb{R}
^{n})$, from (1.5) and Holder inequality, we obtain%
\begin{equation*}
\left\Vert \frac{1}{\lambda _{k}}f(\lambda _{k}u_{k})-\frac{1}{\lambda _{k}}%
f(\lambda _{k}u)\right\Vert _{L^{2}(%
\mathbb{R}
^{n})}
\end{equation*}%
\begin{equation*}
\leq C_{1}\left( \int\nolimits_{%
\mathbb{R}
^{n}}\left\vert u_{k}-u\right\vert ^{2}\left( 1+\left( \lambda
_{k}u_{k}\right) ^{2(p-1)}+\left( \lambda _{k}u\right) ^{2(p-1)}\right)
dx\right) ^{\frac{1}{2}}
\end{equation*}%
\begin{equation*}
\leq C_{2}\left( \left\Vert u_{k}-u\right\Vert _{L^{2}(%
\mathbb{R}
^{n})}+\left\Vert u_{k}-u\right\Vert _{H^{2}(%
\mathbb{R}
^{n})}\left( \left\Vert \lambda _{k}u_{k}\right\Vert _{H^{2}(%
\mathbb{R}
^{n})}^{p-1}+\left\Vert \lambda _{k}u\right\Vert _{H^{2}(%
\mathbb{R}
^{n})}^{p-1}\right) \right)
\end{equation*}%
which yields that%
\begin{equation}
\left\Vert \frac{1}{\lambda _{k}}f(\lambda _{k}u_{k})-\frac{1}{\lambda _{k}}%
f(\lambda _{k}u)\right\Vert _{L^{2}(%
\mathbb{R}
^{n})}\leq C_{3}\left\Vert u_{k}-u\right\Vert _{H^{2}(%
\mathbb{R}
^{n})}\text{,}  \tag{2.3}
\end{equation}%
for $\lambda _{0}\in \lbrack 0,\infty )$, where the positive constant $C_{3}$
depends on $\underset{k}{\sup }\lambda _{k}$ and $\underset{k}{\sup }%
\left\Vert u_{k}\right\Vert _{H^{2}(%
\mathbb{R}
^{n})}$. If, additionally, condition (1.7) is satisfied, then (2.3)
immediatly follows from (1.7), for every $\lambda _{0}\in \left[ 0,\infty %
\right] $. In this case the constant on the right hand side of (2.3) depends
only on $\left\Vert f^{\prime }\right\Vert _{L^{\infty }\left(
\mathbb{R}
\right) }$. Because of the same arguments,%
\begin{equation}
\left\Vert \frac{1}{\lambda _{m}}f(\lambda _{m}u)-\frac{1}{\lambda _{m}}%
f(\lambda _{m}u_{m})\right\Vert _{L^{2}(%
\mathbb{R}
^{n})}\leq C_{4}\left\Vert u_{m}-u\right\Vert _{H^{2}(%
\mathbb{R}
^{n})}  \tag{2.4}
\end{equation}%
holds. We distinguish the following three possibilities for the term $%
\left\Vert \frac{1}{\lambda _{k}}f(\lambda _{k}u)-\frac{1}{\lambda _{m}}%
f(\lambda _{m}u)\right\Vert _{L^{2}(%
\mathbb{R}
^{n})}$.

\textbf{Case 1: }$\lambda _{0}\in \left( 0,\infty \right) $.

By continuity of $f$, we get
\begin{equation*}
\underset{k\rightarrow \infty }{\lim }\frac{1}{\lambda _{k}}f(\lambda _{k}u)=%
\frac{1}{\lambda _{0}}f(\lambda _{0}u)\text{ a.e. in }%
\mathbb{R}
^{n}\text{.}
\end{equation*}%
Since $\left\{ \lambda _{k}\right\} _{k=1}^{\infty }$ is convergent,%
\begin{equation*}
\left\vert \frac{1}{\lambda _{k}}f(\lambda _{k}u)\right\vert \leq
C_{5}\left( \left\vert u\right\vert +\left\vert u\right\vert ^{p}\right)
\end{equation*}%
holds and we deduce%
\begin{equation*}
\left\vert \frac{1}{\lambda _{k}}f(\lambda _{k}u)-\frac{1}{\lambda _{m}}%
f(\lambda _{m}u)\right\vert ^{2}\leq C_{6}(\left\vert u\right\vert
^{2}+\left\vert u\right\vert ^{2p})\text{.}
\end{equation*}%
Since $H^{2}(%
\mathbb{R}
^{n})\subset L^{2p}(%
\mathbb{R}
^{n})$, by Lebesgue dominated convergence theorem, we have%
\begin{equation}
\underset{m\rightarrow \infty }{\lim }\underset{k\rightarrow \infty }{\lim }%
\left\Vert \frac{1}{\lambda _{k}}f(\lambda _{k}u)-\frac{1}{\lambda _{m}}%
f(\lambda _{m}u)\right\Vert _{L^{2}(%
\mathbb{R}
^{n})}^{2}=0\text{.}  \tag{2.5}
\end{equation}%
\textbf{Case 2: }$\lambda _{0}=0$.

Define $\ Q_{1}:=\left\{ x\in
\mathbb{R}
^{n}:u(x)\neq 0\right\} ,$ $Q_{2}:=\{x\in
\mathbb{R}
^{n}:u(x)=0\}$. Then we obtain%
\begin{equation*}
\underset{k\rightarrow \infty }{\lim }\frac{1}{\lambda _{k}}f(\lambda _{k}u)=%
\underset{k\rightarrow \infty }{\lim }\frac{f(\lambda _{k}u)}{\lambda _{k}u}%
u=f^{\prime }(0)u,\text{ a.e. in }Q_{1}
\end{equation*}%
and from (1.6) it follows that%
\begin{equation*}
\underset{k\rightarrow \infty }{\lim }\frac{1}{\lambda _{k}}f(\lambda
_{k}u)=0=f^{\prime }(0)u\text{, in }Q_{2}\text{.}
\end{equation*}%
Similar to case 1, by Lebesgue dominated convergence theorem, we find%
\begin{equation}
\underset{m\rightarrow \infty }{\lim }\underset{k\rightarrow \infty }{\lim }%
\left\Vert \frac{1}{\lambda _{k}}f(\lambda _{k}u)-\frac{1}{\lambda _{m}}%
f(\lambda _{m}u)\right\Vert _{L^{2}(%
\mathbb{R}
^{n})}^{2}=0\text{.}  \tag{2.6}
\end{equation}%
\textbf{Case 3: \ }$\lambda _{0}=\infty $ and, additionally, condition (1.7)
is satisfied.

Define $\widehat{Q}_{1}:=\left\{ x\in
\mathbb{R}
^{n}:u(x)<0\right\} $, $\widehat{Q}_{2}:=\left\{ x\in
\mathbb{R}
^{n}:u(x)>0\right\} $ and $\widehat{Q}_{3}:=\left\{ x\in
\mathbb{R}
^{n}:u(x)=0\right\} $.\newline
Taking into account (1.6)-(1.7), we get%
\begin{equation*}
\underset{k\rightarrow \infty }{\lim }\frac{1}{\lambda _{k}}f(\lambda _{k}u)=%
\underset{k\rightarrow \infty }{\lim }\frac{f(\lambda _{k}u)}{\lambda _{k}u}%
u=\alpha _{1}u,\text{ a.e. in }\widehat{Q}_{1}\text{,}
\end{equation*}%
\begin{equation*}
\underset{k\rightarrow \infty }{\lim }\frac{1}{\lambda _{k}}f(\lambda _{k}u)=%
\underset{k\rightarrow \infty }{\lim }\frac{f(\lambda _{k}u)}{\lambda _{k}u}%
u=\alpha _{2}u,\text{ a.e. in }\widehat{Q}_{2}\text{,}
\end{equation*}%
\begin{equation*}
\underset{k\rightarrow \infty }{\lim }\frac{1}{\lambda _{k}}f(\lambda
_{k}u)=0,\text{ in }\widehat{Q}_{3}\text{.}
\end{equation*}%
Hence, we deduce%
\begin{equation*}
\underset{k\rightarrow \infty }{\lim }\frac{1}{\lambda _{k}}f(\lambda
_{k}u)=\left( \alpha _{1}\chi _{Q_{1}}+\alpha _{2}\chi _{Q_{2}\text{ }%
}\right) u\text{,\ \ a.e. in }%
\mathbb{R}
^{n}\text{.}
\end{equation*}%
Since
\begin{equation*}
\left\vert \frac{1}{\lambda _{k}}f(\lambda _{k}u)-\frac{1}{\lambda _{m}}%
f(\lambda _{m}u)\right\vert ^{2}\leq 2\left( \left\vert \frac{f(\lambda
_{k}u)}{\lambda _{k}}\right\vert ^{2}+\left\vert \frac{f(\lambda _{m}u)}{%
\lambda _{m}}\right\vert ^{2}\right) \leq C_{7}u^{2}\text{, a.e. in }%
\mathbb{R}
^{n}\text{,}
\end{equation*}%
again by Lebesgue dominated convergence theorem,%
\begin{equation}
\underset{m\rightarrow \infty }{\lim }\underset{k\rightarrow \infty }{\lim }%
\left\Vert \frac{1}{\lambda _{k}}f(\lambda _{k}u)-\frac{1}{\lambda _{m}}%
f(\lambda _{m}u)\right\Vert _{L^{2}(%
\mathbb{R}
^{n})}^{2}=0\text{.}  \tag{2.7}
\end{equation}%
Considering (2.2)-(2.7), we obtain%
\begin{equation*}
\underset{m\rightarrow \infty }{\lim \sup }\underset{k\rightarrow \infty }{%
\lim \sup }\left\Vert \frac{1}{\lambda _{k}}f(\lambda _{k}u_{k})-\frac{1}{%
\lambda _{m}}f(\lambda _{m}u_{m})\right\Vert _{L^{2}(%
\mathbb{R}
^{n})}
\end{equation*}%
\begin{equation}
\leq C_{8}\left( \underset{k\rightarrow \infty }{\lim \sup }\left\Vert
u_{k}-u\right\Vert _{H^{2}(%
\mathbb{R}
^{n})}+\underset{m\rightarrow \infty }{\lim \sup }\left\Vert
u_{m}-u\right\Vert _{H^{2}(%
\mathbb{R}
^{n})}\right) \text{.}  \tag{2.8}
\end{equation}%
It is easy to verify that%
\begin{equation*}
\underset{m\rightarrow \infty }{\lim \sup }\underset{k\rightarrow \infty }{%
\lim \sup }\left\Vert u_{k}-u_{m}\right\Vert _{H^{2}(%
\mathbb{R}
^{n})}^{2}=2\underset{k\rightarrow \infty }{\lim \sup }\left\Vert
u_{k}-u\right\Vert _{H^{2}(%
\mathbb{R}
^{n})}^{2}
\end{equation*}%
which, together with (2.8), yields (2.1).
\end{proof}

Let us consider the following problem%
\begin{equation}
\left\{
\begin{array}{c}
u_{\lambda tt}+\Delta ^{2}u_{\lambda }+a(x)u_{\lambda t}+\alpha u_{\lambda
}+\Phi _{\lambda }(u_{\lambda })=0\text{, \ \ }(t,x)\in (0,\infty )\times
\mathbb{R}
^{n}\text{,} \\
u_{\lambda }(0,\cdot )=u_{0\lambda }\in H^{2}(%
\mathbb{R}
^{n})\text{,}\ \ \ \ u_{\lambda t}(0,\cdot )=u_{1\lambda }\in L^{2}(%
\mathbb{R}
^{n})\text{, \ \ \ \ \ \ \ \ \ \ \ \ \ \ \ \ }%
\end{array}%
\right.  \tag{2.9}
\end{equation}%
where $\Phi _{\lambda }(u)=\left\{
\begin{array}{c}
f^{\prime }(0)u\text{, \ \ \ \ \ \ \ \ \ \ \ \ \ if }\lambda =0\text{, \ \ \
} \\
\frac{1}{\lambda }f(\lambda u)\text{, \ \ \ \ \ \ \ if }\lambda \in \left(
0,\infty \right) \text{,} \\
\left\{
\begin{array}{c}
\alpha _{1}u\text{, \ }u\leq 0 \\
\alpha _{2}u\text{, \ }u>0%
\end{array}%
\right. \text{, \ \ if \ }\lambda =\infty \text{ }%
\end{array}%
\right. $. By using semigroup theory, it is easy to show that under
conditions (1.3), (1.5) and (1.6), problem (2.9) generates strongly
continuous semigroup $\left\{ S^{\lambda }(t)\right\} _{t\geq 0}$ in $H^{2}(%
\mathbb{R}
^{n})\times L^{2}(%
\mathbb{R}
^{n})$, for every $\lambda \in \lbrack 0,\infty ]$.

\begin{lemma}
Assume the conditions (1.3), (1.5) and (1.6). If $0<\lambda _{k}\rightarrow
\lambda _{0}\in \lbrack 0,\infty )$ and $\left( u_{0k},u_{1k}\right)
\rightarrow (u_{0},u_{1})$ strongly in $H^{2}(%
\mathbb{R}
^{n})\times L^{2}(%
\mathbb{R}
^{n})$, then we have%
\begin{equation}
S^{\lambda _{k}}(t)(u_{0k},u_{1k})\rightarrow S^{\lambda
_{0}}(t)(u_{0},u_{1})\text{ strongly in }H^{2}(%
\mathbb{R}
^{n})\times L^{2}(%
\mathbb{R}
^{n})\text{.}  \tag{2.10}
\end{equation}%
Furthermore, if, additionally, condition (1.7) is satisfied, then (2.10)
also holds for $\lambda _{0}=\infty $.
\end{lemma}

\begin{proof}
We will establish the following estimates for smooth solutions of (2.9) with
initial data in $H^{4}(%
\mathbb{R}
^{n})\times H^{2}(%
\mathbb{R}
^{n}),~$for which the estimates below are justified. The estimates can be
extended to the weak solutions with initial data in $H^{2}(%
\mathbb{R}
^{n})\times L^{2}(%
\mathbb{R}
^{n})$ by standard density arguments. Let $(u_{k}(t),u_{kt}(t))=S^{\lambda
_{k}}(t)(u_{0k},u_{1k})$. Putting $u_{k}(t)$ and $\lambda _{k}$ instead of $%
u_{\lambda }(t)$ and $\lambda $, respectively, multiplying the obtained
equation by $2u_{kt}$, integrating over $\left( 0,t\right) \times
\mathbb{R}
^{n}$ and taking into account (1.5), we find%
\begin{equation*}
\left\Vert u_{kt}\left( t\right) \right\Vert _{L^{2}(%
\mathbb{R}
^{n})}^{2}+\left\Vert {\small \Delta }u_{k}\left( t\right) \right\Vert
_{L^{2}(%
\mathbb{R}
^{n})}^{2}+\alpha \left\Vert u_{k}\left( t\right) \right\Vert _{L^{2}(%
\mathbb{R}
^{n})}^{2}+\frac{2}{\lambda _{k}^{2}}\int\nolimits_{%
\mathbb{R}
^{n}}F\left( \lambda _{k}u_{k}(t,x)\right) dx+
\end{equation*}%
\begin{equation*}
\text{ }+\int\nolimits_{0}^{t}\int\nolimits_{%
\mathbb{R}
^{n}}a\left( x\right) \left\vert u_{kt}\left( t,x\right) \right\vert
^{2}dx=\left\Vert u_{kt}\left( 0\right) \right\Vert _{L^{2}(%
\mathbb{R}
^{n})}^{2}\text{\ }+\left\Vert {\small \Delta }u_{k}\left( 0\right)
\right\Vert _{L^{2}(%
\mathbb{R}
^{n})}^{2}+\alpha \left\Vert u_{k}\left( 0\right) \right\Vert _{L^{2}(%
\mathbb{R}
^{n})}^{2}
\end{equation*}%
\begin{equation}
+\frac{2}{\lambda _{k}^{2}}\int\nolimits_{%
\mathbb{R}
^{n}}F\left( \lambda _{k}u_{k}(0,x)\right) dx\leq C\left( \left\Vert
u_{1k}\right\Vert _{L^{2}(%
\mathbb{R}
^{n})}^{2}+\left\Vert u_{0k}\right\Vert _{H^{2}(%
\mathbb{R}
^{n})}^{2}+\left\Vert u_{0k}\right\Vert _{H^{2}(%
\mathbb{R}
^{n})}^{p+1}\right) \text{,}  \tag{2.11}
\end{equation}%
where the positive constant $C$ depends on $\underset{k}{\sup }\lambda _{k}$%
. If, additionally, condition (1.7) is satisfied, then the constant $C$
depends only on $\left\Vert f^{\prime }\right\Vert _{L^{\infty }\left(
\mathbb{R}
\right) }$. By using (1.3) and (1.6), we get%
\begin{equation}
\left\Vert u_{kt}\left( t\right) \right\Vert _{L^{2}(%
\mathbb{R}
^{n})}^{2}+\left\Vert u_{k}\left( t\right) \right\Vert _{H^{2}(%
\mathbb{R}
^{n})}^{2}\leq C\left( \left\Vert u_{1k}\right\Vert _{L^{2}(%
\mathbb{R}
^{n})}^{2}+\left\Vert u_{0k}\right\Vert _{H^{2}(%
\mathbb{R}
^{n})}^{2}+\left\Vert u_{0k}\right\Vert _{H^{2}(%
\mathbb{R}
^{n})}^{p+1}\right) \text{.}  \tag{2.12}
\end{equation}%
Since $\left\{ (u_{0k},u_{1k})\right\} _{k=1}^{\infty }$ is convergent, it
is bounded. So $\left\{ u_{k}\right\} _{k=1}^{\infty }$ is bounded in $%
L^{\infty }(0,\infty ;H^{2}(%
\mathbb{R}
^{n}))$ and $\left\{ u_{kt}\right\} _{k=1}^{\infty }$ is bounded in $%
L^{\infty }(0,\infty ;L^{2}(%
\mathbb{R}
^{n})).$ Then, by Banach-Alaoglu theorem, there exist subsequences $\left\{
u_{k_{m}}\right\} _{m=1}^{\infty }$ and $\left\{ u_{k_{m}t}\right\}
_{m=1}^{\infty }$ such that%
\begin{equation}
\left\{
\begin{array}{c}
u_{k_{m}}\rightarrow u\text{ weakly star in \ }L^{\infty }(0,\infty ;H^{2}(%
\mathbb{R}
^{n})) \\
u_{k_{m}t}\rightarrow u_{t}\text{ weakly star in \ }L^{\infty }(0,\infty
;L^{2}(%
\mathbb{R}
^{n}))%
\end{array}%
\right. \text{,}  \tag{2.13}
\end{equation}%
which yields the boundedness of the sequence $\left\{ u_{k_{m}}\right\}
_{m=1}^{\infty }$ in $H^{1}\left( \left( 0,\infty \right) \times
\mathbb{R}
^{n}\right) $. Then for any $r>0$ and $T>0$, by using the compact embedding $%
H^{1}\left( \left( 0,T\right) \times B(0,r)\right) \hookrightarrow
L^{2}\left( \left( 0,T\right) \times B(0,r)\right) $, we have%
\begin{equation*}
u_{k_{m}}\rightarrow u\text{ strongly in }L^{2}\left( \left( 0,T\right)
\times B(0,r)\right) \text{,}
\end{equation*}%
where, $B(0,r)=\{x\in
\mathbb{R}
^{n}:\left\vert x\right\vert <r\}$. Hence, there exists a subsequence $%
\left\{ u_{k_{m_{l}}}\right\} _{l=1}^{\infty }\subset \left\{
u_{k_{m}}\right\} _{m=1}^{\infty }$ such that $u_{k_{m_{l}}}(t,x)\rightarrow
u(t,x)$ a.e. in $\left( 0,T\right) \times B(0,r)$. Then, by using the same
arguments in previous lemma, we obtain%
\begin{equation*}
\frac{1}{\lambda _{k_{m_{l}}}}f(\lambda
_{k_{m_{l}}}u_{k_{m_{l}}}(t,x))\rightarrow \Phi _{\lambda _{0}}(u(t,x))\text{
a.e. in }\left( 0,T\right) \times B(0,r)
\end{equation*}%
and, since, by (1.5), the sequence $\left\{ \frac{1}{\lambda _{k_{m}}}%
f(\lambda _{k_{m}}u_{k_{m}})\right\} _{m=1}^{\infty }$ is bounded in $%
L^{2}((0,\infty )\times
\mathbb{R}
^{n})$, we get
\begin{equation}
\frac{1}{\lambda _{k_{m_{l}}}}f(\lambda
_{k_{m_{l}}}u_{k_{m_{l}}})\rightarrow \Phi _{\lambda _{0}}(u)\text{ weakly
in }L^{2}\left( \left( 0,T\right) \times B(0,r)\right) \text{.}  \tag{2.14}
\end{equation}%
Furthermore, by (2.9)$_{1}$ and (2.12)-(2.14), the sequence $\left\{
u_{k_{m_{l}}tt}\right\} _{l=1}^{\infty }$ is bounded in $L^{\infty
}(0,\infty ;H^{-2}(%
\mathbb{R}
^{n}))$, so we have%
\begin{equation}
u_{k_{m_{l}}tt}\rightarrow u_{tt}\text{ weakly star in \ }L^{\infty
}(0,\infty ;H^{-2}(%
\mathbb{R}
^{n}))\text{.}  \tag{2.15}
\end{equation}%
From (2.12)-(2.15), we obtain that $u(t,x)$ is a solution of problem (2.9).
By the uniqueness of solutions, we deduce%
\begin{equation*}
S^{\lambda _{k_{m_{l}}}}(t)(u_{0k_{m_{l}}},u_{1k_{m_{l}}})\rightarrow
S^{\lambda _{0}}(t)(u_{0},u_{1})\text{ weakly in }H^{2}(%
\mathbb{R}
^{n})\times L^{2}(%
\mathbb{R}
^{n})\text{.}
\end{equation*}%
Similarly, one can show that every subsequence of $\left\{ u_{k}\right\}
_{k=1}^{\infty }$ has a further subsequence which is weakly convergent to $u$%
. It means that%
\begin{equation}
S^{\lambda _{k}}(t)(u_{0k},u_{1k})\rightarrow S^{\lambda
_{0}}(t)(u_{0},u_{1})\text{ weakly in }H^{2}(%
\mathbb{R}
^{n})\times L^{2}(%
\mathbb{R}
^{n})\text{.}  \tag{2.16}
\end{equation}%
Multiplying the equation%
\begin{equation*}
u_{ktt}-u_{mtt}+\Delta ^{2}\left( u_{k}-u_{m}\right) +a\left( x\right)
\left( u_{kt}-u_{mt}\right) +\alpha \left( u_{k}-u_{m}\right) +\frac{1}{%
\lambda _{k}}f\left( \lambda _{k}u_{k}\right) -\frac{1}{\lambda _{m}}f\left(
\lambda _{m}u_{m}\right) =0
\end{equation*}%
by $2(u_{kt}-u_{mt})$, integrating over $\left( 0,t\right) \times
\mathbb{R}
^{n}$ and considering (1.3), we have%
\begin{eqnarray*}
&&\left\Vert u_{kt}\left( t\right) -u_{mt}\left( t\right) \right\Vert
_{L^{2}(%
\mathbb{R}
^{n})}^{2}+\left\Vert \Delta u_{k}\left( t\right) -\Delta u_{m}\left(
t\right) \right\Vert _{L^{2}(%
\mathbb{R}
^{n})}^{2}+\alpha \left\Vert u_{k}\left( t\right) -u_{m}\left( t\right)
\right\Vert _{L^{2}(%
\mathbb{R}
^{n})}^{2} \\
&\leq &\left\Vert u_{1k}-u_{1m}\right\Vert _{L^{2}(%
\mathbb{R}
^{n})}^{2}+\left\Vert \Delta u_{0k}-\Delta u_{0m}\right\Vert _{L^{2}(%
\mathbb{R}
^{n})}^{2}+\alpha \left\Vert u_{0k}-u_{0m}\right\Vert _{L^{2}(%
\mathbb{R}
^{n})}^{2} \\
&&+\int\nolimits_{0}^{t}\left( \left\Vert u_{kt}(s)-u_{mt}(s)\right\Vert
_{L^{2}(%
\mathbb{R}
^{n})}^{2}+\left\Vert \frac{1}{\lambda _{k}}f(\lambda _{k}u_{k}(s))-\frac{1}{%
\lambda _{m}}f(\lambda _{m}u_{m}(s))\right\Vert _{L^{2}(%
\mathbb{R}
^{n})}^{2}\right) ds\text{.}
\end{eqnarray*}%
From above inequality and previous lemma, we obtain%
\begin{eqnarray*}
&&\underset{m\rightarrow \infty }{\lim \sup }\underset{k\rightarrow \infty }{%
\lim \sup }\left( \left\Vert u_{kt}(t)-u_{mt}(t)\right\Vert _{L^{2}(%
\mathbb{R}
^{n})}^{2}+\left\Vert u_{k}(t)-u_{m}(t)\right\Vert _{H^{2}(%
\mathbb{R}
^{n})}^{2}\right) \\
&\leq &C\underset{m\rightarrow \infty }{\lim \sup }\underset{k\rightarrow
\infty }{\lim \sup }\int\nolimits_{0}^{t}\left( \left\Vert
u_{kt}(s)-u_{mt}(s)\right\Vert _{L^{2}(%
\mathbb{R}
^{n})}^{2}+\left\Vert u_{k}(s)-u_{m}(s)\right\Vert _{H^{2}(%
\mathbb{R}
^{n})}^{2}\right) ds\text{.}
\end{eqnarray*}%
Since $\left\{ \left( u_{k},u_{kt}\right) \right\} _{k=1}^{\infty }$ is
bounded in $L^{\infty }(0,\infty ;H^{2}\left(
\mathbb{R}
^{n}\right) \times L^{2}\left(
\mathbb{R}
^{n}\right) )$, by reverse Fatou's lemma, we get
\begin{eqnarray*}
&&\underset{m\rightarrow \infty }{\lim \sup }\underset{k\rightarrow \infty }{%
\lim \sup }\left( \left\Vert u_{kt}(t)-u_{mt}(t)\right\Vert _{L^{2}(%
\mathbb{R}
^{n})}^{2}+\left\Vert u_{k}(t)-u_{m}(t)\right\Vert _{H^{2}(%
\mathbb{R}
^{n})}^{2}\right) \\
&\leq &C\int\nolimits_{0}^{t}\underset{m\rightarrow \infty }{\lim \sup }%
\underset{k\rightarrow \infty }{\lim \sup }\left( \left\Vert
u_{kt}(s)-u_{mt}(s)\right\Vert _{L^{2}(%
\mathbb{R}
^{n})}^{2}+\left\Vert u_{k}(s)-u_{m}(s)\right\Vert _{H^{2}(%
\mathbb{R}
^{n})}^{2}\right) ds\text{.}
\end{eqnarray*}%
Hence, by Gronwall's inequality,%
\begin{equation*}
\underset{m\rightarrow \infty }{\lim \sup }\underset{k\rightarrow \infty }{%
\lim \sup }\left( \left\Vert u_{kt}(t)-u_{mt}(t)\right\Vert _{L^{2}(%
\mathbb{R}
^{n})}^{2}+\left\Vert u_{k}(t)-u_{m}(t)\right\Vert _{H^{2}(%
\mathbb{R}
^{n})}^{2}\right) =0\text{.}
\end{equation*}%
So, $\left\{ \left( u_{k}\left( t\right) ,u_{kt}\left( t\right) \right)
\right\} _{k=1}^{\infty }$ is a Cauchy \ subsequence in $H^{2}\left(
\mathbb{R}
^{n}\right) \times L^{2}\left(
\mathbb{R}
^{n}\right) $, which together with (2.16) yields (2.10).
\end{proof}

\begin{lemma}
Assume that conditions (1.3)-(1.6) hold and $B$ is a bounded subset of$\
H^{2}\left(
\mathbb{R}
^{n}\right) \times L^{2}\left(
\mathbb{R}
^{n}\right) $. Then for every \ $M>0$ the sequence of the form $\left\{
S^{\lambda _{k}}(t_{k})\varphi _{k}\right\} _{k=1}^{\infty }$, $\left\{
\varphi _{k}\right\} _{k=1}^{\infty }\subset B$, $t_{k}\rightarrow \infty $,
$\left\{ \lambda _{k}\right\} _{k=1}^{\infty }\subset \left( 0,M\right] $,
is relatively compact in $H^{2}\left(
\mathbb{R}
^{n}\right) \times L^{2}\left(
\mathbb{R}
^{n}\right) $. Furthermore, if, additionally, condition (1.7) is satisfied,
then the sequence of the form $\left\{ S^{\lambda _{k}}(t_{k})\varphi
_{k}\right\} _{k=1}^{\infty }$, $\left\{ \varphi _{k}\right\} _{k=1}^{\infty
}\subset B$, $t_{k}\rightarrow \infty $, $\lambda _{k}\rightarrow \infty $,
is also relatively compact in $H^{2}\left(
\mathbb{R}
^{n}\right) \times L^{2}\left(
\mathbb{R}
^{n}\right) $.
\end{lemma}

\begin{proof}
Since $\left\{ \varphi _{k}\right\} _{k=1}^{\infty }$ is bounded in $%
H^{2}\left(
\mathbb{R}
^{n}\right) \times L^{2}\left(
\mathbb{R}
^{n}\right) $, under conditions of lemma, from (2.12) it follows that the
sequence $\left\{ S^{\lambda _{k}}\left( .\right) \varphi _{k}\right\}
_{k=1}^{\infty }$ is bounded in $C_{b}\left( 0,\infty ;H^{2}\left(
\mathbb{R}
^{n}\right) \times L^{2}\left(
\mathbb{R}
^{n}\right) \right) $, where $C_{b}\left( 0,\infty ;H^{2}\left(
\mathbb{R}
^{n}\right) \times L^{2}\left(
\mathbb{R}
^{n}\right) \right) $ is the space of continuously bounded functions from $%
\left[ 0,\infty \right) $ to $H^{2}\left(
\mathbb{R}
^{n}\right) \times L^{2}\left(
\mathbb{R}
^{n}\right) $. Then for any $T_{0}\geq 0$ \ there exists a subsequence $%
\left\{ k_{m}\right\} _{m=1}^{\infty }$ such that $t_{k_{m}}\geq T_{0}$, and
\begin{equation}
\left\{
\begin{array}{c}
\lambda _{k_{m}}\rightarrow \lambda _{0}\text{ in }\overline{%
\mathbb{R}
}\text{,} \\
S^{\lambda _{k_{m}}}(t_{k_{m}}-T_{0})\varphi _{k_{m}}\rightarrow \varphi _{0}%
\text{ weakly in }H^{2}\left(
\mathbb{R}
^{n}\right) \times L^{2}\left(
\mathbb{R}
^{n}\right) \text{,} \\
v_{m}\rightarrow v\text{ weakly star in }L^{\infty }\left( 0,\infty
;H^{2}\left(
\mathbb{R}
^{n}\right) \right) \text{,} \\
v_{mt}\rightarrow v_{t}\text{ weakly star in }L^{\infty }\left( 0,\infty
;L^{2}\left(
\mathbb{R}
^{n}\right) \right) \text{,} \\
v_{m}(t)\rightarrow v(t)\text{ weakly in }H^{2}\left(
\mathbb{R}
^{n}\right) \text{, }\forall t\geq 0\text{,}%
\end{array}%
\right.   \tag{2.17}
\end{equation}%
for some $\lambda _{0}\in \left[ 0,\infty \right] $, $\varphi _{0}\in
H^{2}\left(
\mathbb{R}
^{n}\right) \times L^{2}\left(
\mathbb{R}
^{n}\right) $ and $v\in L^{\infty }\left( 0,\infty ;H^{2}\left(
\mathbb{R}
^{n}\right) \right) \cap W^{1,\infty }\left( 0,\infty ;L^{2}\left(
\mathbb{R}
^{n}\right) \right) $, where $\left( v_{m}(t\right) ,v_{mt}\left( t\right)
)=S^{\lambda _{k_{m}}}(t+t_{k_{m}}-T_{0})\varphi _{k_{m}}$ and $\overline{%
\mathbb{R}
}$ is the extended set of real numbers. \newline
Taking into account (2.11), we get%
\begin{equation}
\int\nolimits_{0}^{T}\left\Vert v_{mt}(t)\right\Vert _{L^{2}\left(
\mathbb{R}
^{n}\smallsetminus B\left( 0,r_{0}\right) \right) }^{2}dt\leq c_{1}\text{, }%
\forall T\geq 0\text{.}  \tag{2.18}
\end{equation}%
By (2.9)$_{1}$, we have%
\begin{equation*}
v_{mtt}+\Delta ^{2}v_{m}+a(x)v_{mt}+\alpha v_{m}+\frac{1}{\lambda _{k_{m}}}%
f\left( \lambda _{k_{m}}v_{m}\right) =0\text{.}
\end{equation*}%
Let $\eta \in C^{\infty }\left(
\mathbb{R}
^{n}\right) $, $0\leq \eta \left( x\right) \leq 1$, $\eta \left( x\right)
=\left\{
\begin{array}{c}
0,\text{ }\left\vert x\right\vert \leq 1 \\
1,\text{ }\left\vert x\right\vert \geq 2%
\end{array}%
\right. $ and $\eta _{r}\left( x\right) =\eta \left( \frac{x}{r}\right) $.
Multiplying above equation by $\eta _{r}^{2}v_{m}$, integrating over $\left(
0,T\right) \times
\mathbb{R}
^{n}$ and taking into account (1.6), we find%
\begin{equation*}
\int\nolimits_{0}^{T}\left( \left\Vert \eta _{r}\Delta v_{m}(t)\right\Vert
_{L^{2}\left(
\mathbb{R}
^{n}\right) }^{2}+\alpha \left\Vert \eta _{r}v_{m}(t)\right\Vert
_{L^{2}\left(
\mathbb{R}
^{n}\right) }^{2}\right) dt
\end{equation*}%
\begin{equation*}
\leq \int\nolimits_{0}^{T}\left\Vert \eta _{r}v_{mt}(t)\right\Vert
_{L^{2}\left(
\mathbb{R}
^{n}\right) }^{2}dt-\left. \left( \int\nolimits_{%
\mathbb{R}
^{n}}\eta _{r}^{2}\left( x\right) v_{mt}\left( t,x\right) v_{m}\left(
t,x\right) dx\right) \right\vert _{0}^{T}
\end{equation*}%
\begin{equation*}
-\frac{4}{r}\sum\nolimits_{i=1}^{n}\int\nolimits_{0}^{T}\int\nolimits_{%
\mathbb{R}
^{n}}\eta _{r}\left( x\right) \eta _{x_{i}}(\frac{x}{r})\Delta v_{m}\left(
t,x\right) v_{mx_{i}}\left( t,x\right) dxdt
\end{equation*}%
\begin{equation*}
-\int\nolimits_{0}^{T}\int\nolimits_{%
\mathbb{R}
^{n}}\Delta \left( \eta _{r}^{2}\left( x\right) \right) \Delta v_{m}\left(
t,x\right) v_{m}\left( t,x\right) dxdt-\frac{1}{2}\left. \left(
\int\nolimits_{%
\mathbb{R}
^{n}}\eta _{r}^{2}\left( x\right) a\left( x\right) v_{m}^{2}\left(
t,x\right) dx\right) \right\vert _{0}^{T}\text{.}
\end{equation*}%
Considering (2.17) and (2.18), we get%
\begin{equation*}
\underset{m\rightarrow \infty }{\lim \sup }\int\nolimits_{0}^{T}\left(
\left\Vert \eta _{r}\Delta v_{m}(t)\right\Vert _{L^{2}\left(
\mathbb{R}
^{n}\right) }^{2}+\alpha \left\Vert \eta _{r}v_{m}(t)\right\Vert
_{L^{2}\left(
\mathbb{R}
^{n}\right) }^{2}\right) dt
\end{equation*}%
\begin{equation}
\leq c_{2}\left( 1+\frac{T}{r}\right) \text{, }\forall T\geq 0\text{ and }%
\forall r\geq r_{0}\text{.}  \tag{2.19}
\end{equation}%
By (2.9)$_{1}$, we also have%
\begin{equation*}
v_{mtt}-v_{ltt}+\Delta ^{2}\left( v_{m}-v_{l}\right) +a(x)\left(
v_{mt}-v_{lt}\right) +\alpha \left( v_{m}-v_{l}\right)
\end{equation*}%
\begin{equation}
+\frac{1}{\lambda _{k_{m}}}f\left( \lambda _{k_{m}}v_{m}\right) -\frac{1}{%
\lambda _{k_{l}}}f\left( \lambda _{k_{l}}v_{l}\right) =0\text{.}  \tag{2.20}
\end{equation}%
Now, multiplying (2.20) by $\sum\nolimits_{i=1}^{n}x_{i}\left( 1-\eta
_{2r}\right) \left( v_{m}-v_{l}\right) _{x_{i}}+\frac{1}{2}(n-1)\left(
1-\eta _{2r}\right) \left( v_{m}-v_{l}\right) $, and integrating over $%
\left( 0,T\right) \times
\mathbb{R}
^{n}$, we obtain%
\begin{equation*}
\frac{3}{2}\int\nolimits_{0}^{T}\left\Vert \Delta \left( v_{m}\left(
t\right) -v_{l}\left( t\right) \right) \right\Vert _{L^{2}\left( B\left(
0,2r\right) \right) }^{2}dt+\frac{1}{2}\int\nolimits_{0}^{T}\left\Vert
v_{mt}\left( t\right) -v_{lt}\left( t\right) \right\Vert _{L^{2}\left(
B\left( 0,2r\right) \right) }^{2}dt
\end{equation*}%
\begin{equation*}
\leq \left\vert \sum\nolimits_{i=1}^{n}\left( \int\nolimits_{B\left(
0,4r\right) }\left( 1-\eta _{2r}\left( x\right) \right) x_{i}\left(
v_{m}(T,x)-v_{l}(T,x)\right) _{x_{i}}\left( v_{mt}(T,x)-v_{lt}(T,x)\right)
dx\right) \right\vert
\end{equation*}%
\begin{equation*}
+\left\vert \sum\nolimits_{i=1}^{n}\left( \int\nolimits_{B\left( 0,4r\right)
}\left( 1-\eta _{2r}\left( x\right) \right) x_{i}\left(
v_{m}(0,x)-v_{l}(0,x)\right) _{x_{i}}\left( v_{mt}(0,x)-v_{lt}(0,x)\right)
dx\right) \right\vert
\end{equation*}%
\begin{equation*}
+\frac{1}{2}(n-1)\left\vert \left( \int\nolimits_{B\left( 0,4r\right)
}\left( 1-\eta _{2r}\left( x\right) \right) \left(
v_{mt}(T,x)-v_{lt}(T,x)\right) \left( v_{m}(T,x)-v_{l}(T,x)\right) dx\right)
\right\vert
\end{equation*}%
\begin{equation*}
+\frac{1}{2}(n-1)\left\vert \left( \int\nolimits_{B\left( 0,4r\right)
}\left( 1-\eta _{2r}\left( x\right) \right) \left(
v_{mt}(0,x)-v_{lt}(0,x)\right) \left( v_{m}(0,x)-v_{l}(0,x)\right) dx\right)
\right\vert
\end{equation*}%
\begin{equation*}
+\frac{1}{4r}\left\vert
\sum\nolimits_{i,j=1}^{n}\int\nolimits_{0}^{T}\int\nolimits_{B\left(
0,4r\right) \backslash B\left( 0,2r\right) }\eta _{x_{i}}\left( \frac{x}{2r}%
\right) x_{i}(v_{mt}\left( t,x\right) -v_{lt}\left( t,x\right)
)^{2}dxdt\right\vert
\end{equation*}%
\begin{equation*}
+\frac{1}{4r}\left\vert
\sum\nolimits_{i,j=1}^{n}\int\nolimits_{0}^{T}\int\nolimits_{B\left(
0,4r\right) \backslash B\left( 0,2r\right) }\eta _{x_{i}}\left( \frac{x}{2r}%
\right) x_{i}(\Delta v_{m}\left( t,x\right) -\Delta v_{l}\left( t,x\right)
)^{2}dxdt\right\vert
\end{equation*}%
\begin{equation*}
+\left\vert
\sum\nolimits_{i=1}^{n}\int\nolimits_{0}^{T}\int\nolimits_{B\left(
0,4r\right) }\Delta \left( \left( 1-\eta _{2r}\left( x\right) \right)
x_{i}\right) \left( v_{m}\left( t,x\right) -v_{l}\left( t,x\right) \right)
_{x_{i}}\Delta \left( v_{m}\left( t,x\right) -v_{l}\left( t,x\right) \right)
dxdt\right\vert
\end{equation*}%
\begin{equation*}
+\frac{1}{r}\left\vert
\sum\nolimits_{i,j=1}^{n}\int\nolimits_{0}^{T}\int\nolimits_{B\left(
0,4r\right) \backslash B\left( 0,2r\right) }\eta _{x_{j}}\left( \frac{x}{2r}%
\right) x_{i}\left( v_{m}\left( t,x\right) -v_{l}\left( t,x\right) \right)
_{x_{i}x_{j}}\Delta \left( v_{m}\left( t,x\right) -v_{l}\left( t,x\right)
\right) dxdt\right\vert
\end{equation*}%
\begin{equation*}
+\frac{1}{2}(n-1)\left\vert \int\nolimits_{0}^{T}\int\nolimits_{B\left(
0,4r\right) }\Delta \left( 1-\eta _{2r}\left( x\right) \right) \left(
v_{m}\left( t,x\right) -v_{l}\left( t,x\right) \right) \Delta \left(
v_{m}\left( t,x\right) -v_{l}\left( t,x\right) \right) dxdt\right\vert
\end{equation*}%
\begin{equation*}
+\frac{1}{2r}(n-1)\left\vert
\sum\nolimits_{i=1}^{n}\int\nolimits_{0}^{T}\int\nolimits_{B\left(
0,4r\right) }\eta _{x_{i}}\left( \frac{x}{2r}\right) \left( v_{m}\left(
t,x\right) -v_{l}\left( t,x\right) \right) _{x_{i}}\Delta \left( v_{m}\left(
t,x\right) -v_{l}\left( t,x\right) \right) dxdt\right\vert
\end{equation*}%
\begin{equation*}
+\left\vert
\sum\nolimits_{i=1}^{n}\int\nolimits_{0}^{T}\int\nolimits_{B\left(
0,4r\right) }\left( 1-\eta _{2r}\left( x\right) \right) x_{i}\left(
v_{m}\left( t,x\right) -v_{l}\left( t,x\right) \right) _{x_{i}}a\left(
x\right) \left( v_{mt}\left( t,x\right) -v_{lt}\left( t,x\right) \right)
dxdt\right\vert
\end{equation*}%
\begin{equation*}
+\frac{1}{2}(n-1)\left\vert \int\nolimits_{0}^{T}\int\nolimits_{B\left(
0,4r\right) }\left( 1-\eta _{2r}\left( x\right) \right) \left( v_{m}\left(
t,x\right) -v_{l}\left( t,x\right) \right) a\left( x\right) \left(
v_{mt}\left( t,x\right) -v_{lt}\left( t,x\right) \right) dxdt\right\vert
\end{equation*}%
\begin{equation*}
+\alpha \left\vert
\sum\nolimits_{i=1}^{n}\int\nolimits_{0}^{T}\int\nolimits_{B\left(
0,4r\right) }\left( 1-\eta _{2r}\left( x\right) \right) x_{i}\left(
v_{m}\left( t,x\right) -v_{l}\left( t,x\right) \right) _{x_{i}}\left(
v_{m}\left( t,x\right) -v_{l}\left( t,x\right) \right) dxdt\right\vert
\end{equation*}%
\begin{equation*}
+\left\vert
\sum\nolimits_{i=1}^{n}\int\nolimits_{0}^{T}\int\nolimits_{B\left(
0,4r\right) }\left( 1-\eta _{2r}\left( x\right) \right) x_{i}\left(
v_{m}\left( t,x\right) -v_{l}\left( t,x\right) \right) _{x_{i}}\right.
\end{equation*}%
\begin{equation*}
\left. \left( \frac{1}{\lambda _{k_{m}}}f\left( \lambda _{k_{m}}v_{m}\left(
t,x\right) \right) -\frac{1}{\lambda _{k_{l}}}f\left( \lambda
_{k_{l}}v_{l}\left( t,x\right) \right) \right) dxdt\right\vert
\end{equation*}%
\begin{equation*}
+\frac{1}{2}(n-1)\left\vert \int\nolimits_{0}^{T}\int\nolimits_{B\left(
0,4r\right) }\left( 1-\eta _{2r}\left( x\right) \right) \left( v_{m}\left(
t,x\right) -v_{l}\left( t,x\right) \right) \right.
\end{equation*}%
\begin{equation*}
\left. \left( \frac{1}{\lambda _{k_{m}}}f\left( \lambda _{k_{m}}v_{m}\left(
t,x\right) \right) -\frac{1}{\lambda _{k_{l}}}f\left( \lambda
_{k_{l}}v_{l}\left( t,x\right) \right) \right) dxdt\right\vert
\end{equation*}%
\begin{equation*}
\leq c_{3}r\left( \left\Vert \nabla v_{m}\left( T\right) -\nabla v_{l}\left(
T\right) \right\Vert _{L^{2}(B\left( 0,4r\right) )}+\left\Vert \nabla
v_{m}\left( 0\right) -\nabla v_{l}\left( 0\right) \right\Vert
_{L^{2}(B\left( 0,4r\right) )}\right)
\end{equation*}%
\begin{equation*}
+c_{3}\left\Vert v_{mt}-v_{lt}\right\Vert _{L^{2}(0,T;L^{2}(B\left(
0,4r\right) \backslash B\left( 0,2r\right) ))}^{2}+c_{3}\left\Vert
v_{m}-v_{l}\right\Vert _{L^{2}(0,T;H^{2}(B\left( 0,4r\right) \backslash
B\left( 0,2r\right) ))}^{2}
\end{equation*}%
\begin{equation*}
+c_{3}r\sqrt{T}\left\Vert \nabla v_{m}-\nabla v_{l}\right\Vert
_{L^{2}(\left( 0,T\right) \times B\left( 0,4r\right) )}\text{,}
\end{equation*}%
for $\left\{ \lambda _{k_{m}}\right\} _{m=1}^{\infty }\subset \left( 0,M%
\right] $, where the constant $c_{3}$ depends on $\left\Vert \eta
\right\Vert _{C^{2}(\overline{B(0,2)})}$, $M$, $\underset{m}{\sup }%
\left\Vert v_{mt}\right\Vert _{L^{\infty }(0,\infty ;L^{2}(%
\mathbb{R}
^{n}))}$ and $\underset{m}{\sup }\left\Vert v_{m}\right\Vert _{L^{\infty
}(0,\infty ;H^{2}(%
\mathbb{R}
^{n}))}$. If, additionally, condition (1.7) is satisfied, then the above
inequality holds for $\left\{ \lambda _{k_{m}}\right\} _{m=1}^{\infty
}\subset (0,\infty ),$ with the constant $c_{3}$ depending on $\left\Vert
\eta \right\Vert _{C^{2}(\overline{B(0,2)})}$, $\left\Vert f^{\prime
}\right\Vert _{L^{\infty }\left(
\mathbb{R}
\right) }$ , $\underset{m}{\sup }\left\Vert v_{mt}\right\Vert _{L^{\infty
}(0,\infty ;L^{2}(%
\mathbb{R}
^{n}))}$ and $\underset{m}{\sup }\left\Vert v_{m}\right\Vert _{L^{\infty
}(0,\infty ;H^{2}(%
\mathbb{R}
^{n}))}$. Since the sequence $\left\{ v_{m}\right\} _{m=1}^{\infty }$ is
bounded in $C\left( \left[ 0,T\right] ;H^{2}\left(
\mathbb{R}
^{n}\right) \right) $ and the sequence $\left\{ v_{mt}\right\}
_{m=1}^{\infty }$ is bounded in $C\left( \left[ 0,T\right] ;L^{2}\left(
\mathbb{R}
^{n}\right) \right) $, by generalized Arzela-Ascoli theorem, the sequence $%
\left\{ v_{m}\right\} _{m=1}^{\infty }$ is relatively compact in $C\left( %
\left[ 0,T\right] ;H^{1}\left( B\left( 0,r\right) \right) \right) $, for
every $r>0$. Then taking into account (2.17)-(2.19) in the last inequality,
we have
\begin{equation*}
\underset{m\rightarrow \infty }{\lim \sup }\underset{l\rightarrow \infty }{%
\lim \sup }\int\nolimits_{0}^{T}\left[ \left\Vert \Delta \left( v_{m}\left(
t\right) -v_{l}\left( t\right) \right) \right\Vert _{L^{2}\left( B\left(
0,2r\right) \right) }^{2}+\left\Vert v_{mt}\left( t\right) -v_{lt}\left(
t\right) \right\Vert _{L^{2}\left( B\left( 0,2r\right) \right) }^{2}\right]
dt
\end{equation*}%
\begin{equation*}
\leq c_{4}\left( 1+\frac{T}{r}\right) \text{, }\forall T\geq 0\text{ and }%
\forall r\geq r_{0}\text{,}
\end{equation*}%
which, again together with (2.17)-(2.19), yields%
\begin{equation*}
\underset{m\rightarrow \infty }{\lim \sup }\underset{l\rightarrow \infty }{%
\lim \sup }\int\nolimits_{0}^{T}\left[ \left\Vert v_{m}\left( t\right)
-v_{l}\left( t\right) \right\Vert _{H^{2}\left(
\mathbb{R}
^{n}\right) }^{2}+\left\Vert v_{mt}\left( t\right) -v_{lt}\left( t\right)
\right\Vert _{L^{2}\left(
\mathbb{R}
^{n}\right) }^{2}\right] dt
\end{equation*}%
\begin{equation}
\leq c_{5}\left( 1+\frac{T}{r}\right) \text{, }\forall T\geq 0\text{ and }%
\forall r\geq r_{0}\text{. }  \tag{2.21}
\end{equation}%
Multiplying (2.20) by $2\left( v_{mt}-v_{lt}\right) $ integrating over $%
\left( t,T\right) \times
\mathbb{R}
^{n}$ and considering (1.3), we get%
\begin{equation*}
\left\Vert v_{mt}\left( T\right) -v_{lt}\left( T\right) \right\Vert
_{L^{2}\left(
\mathbb{R}
^{n}\right) }^{2}+\left\Vert \Delta \left( v_{m}\left( T\right) -v_{l}\left(
T\right) \right) \right\Vert _{L^{2}\left(
\mathbb{R}
^{n}\right) }^{2}+\alpha \left\Vert v_{m}\left( T\right) -v_{l}\left(
T\right) \right\Vert _{L^{2}\left(
\mathbb{R}
^{n}\right) }^{2}
\end{equation*}%
\begin{equation*}
\leq \left\Vert v_{mt}\left( t\right) -v_{lt}\left( t\right) \right\Vert
_{L^{2}\left(
\mathbb{R}
^{n}\right) }^{2}+\left\Vert \Delta \left( v_{m}\left( t\right) -v_{l}\left(
t\right) \right) \right\Vert _{L^{2}\left(
\mathbb{R}
^{n}\right) }^{2}+\alpha \left\Vert v_{m}\left( t\right) -v_{l}\left(
t\right) \right\Vert _{L^{2}\left(
\mathbb{R}
^{n}\right) }^{2}
\end{equation*}%
\begin{equation*}
+2\int\nolimits_{t}^{T}\int\nolimits_{%
\mathbb{R}
^{n}}\left( \frac{1}{\lambda _{k_{l}}}f\left( \lambda _{k_{l}}v_{l}\left(
s,x\right) \right) -\frac{1}{\lambda _{k_{m}}}f\left( \lambda
_{k_{m}}v_{m}\left( s,x\right) \right) \right) \left( v_{mt}\left(
s,x\right) -v_{lt}\left( s,x\right) \right) dxds\text{.}
\end{equation*}%
Integrating the last inequality from $0$ to $T$ with respect to $t$ and
taking into account (2.21), we have%
\begin{equation*}
\underset{m\rightarrow \infty }{\lim \sup }\underset{l\rightarrow \infty }{%
\lim \sup }\left( \left\Vert v_{mt}\left( T\right) -v_{lt}\left( T\right)
\right\Vert _{L^{2}\left(
\mathbb{R}
^{n}\right) }^{2}+\left\Vert \Delta \left( v_{m}\left( T\right) -v_{l}\left(
T\right) \right) \right\Vert _{L^{2}\left(
\mathbb{R}
^{n}\right) }^{2}\right.
\end{equation*}%
\begin{equation*}
\left. +\alpha \left\Vert v_{m}\left( T\right) -v_{l}\left( T\right)
\right\Vert _{L^{2}\left(
\mathbb{R}
^{n}\right) }^{2}\right) \leq c_{6}(\frac{1}{T}+\frac{1}{r})
\end{equation*}%
\begin{equation*}
+\frac{1}{T}\underset{m\rightarrow \infty }{\lim \sup }\underset{%
l\rightarrow \infty }{\lim \sup }\int\nolimits_{0}^{T}\int\nolimits_{t}^{T}%
\int\nolimits_{%
\mathbb{R}
^{n}}\left( \frac{1}{\lambda _{k_{l}}}f\left( \lambda _{k_{l}}v_{l}\left(
s,x\right) \right) -\frac{1}{\lambda _{k_{m}}}f\left( \lambda
_{k_{m}}v_{m}\left( s,x\right) \right) \right)
\end{equation*}%
\begin{equation}
\times \left( v_{mt}\left( s,x\right) -v_{lt}\left( s,x\right) \right) dxdsdt%
\text{, }  \tag{2.22}
\end{equation}%
for all $T\geq 1$. Let us estimate the second term on the right side of
(2.22). By (1.7) and (2.17), we find%
\begin{equation*}
\frac{1}{\lambda _{k_{m}}^{2}}F\left( \lambda _{k_{m}}v_{m}\left( t,x\right)
\right) \rightarrow \Psi _{\lambda _{0}}\left( v\left( t,x\right) \right)
\text{ a.e. in }\left( 0,T\right) \times B\left( 0,r\right) \text{, }\forall
r>0\text{,}
\end{equation*}%
where $\Psi _{\lambda }\left( s\right) =\int\nolimits_{0}^{s}\Phi _{\lambda
}\left( \tau \right) d\tau $. On the other hand, since $\left\{ \frac{1}{%
\lambda _{k_{m}}^{2}}F\left( \lambda _{k_{m}}v_{m}\right) \right\}
_{m=1}^{\infty }$ is bounded \newline
in $W^{1,1}\left( \left( 0,T\right) \times
\mathbb{R}
^{n}\right) $, we obtain%
\begin{equation}
\left\{
\begin{array}{c}
\frac{1}{\lambda _{k_{m}}^{2}}F\left( \lambda _{k_{m}}v_{m}\right)
\rightarrow \Psi _{\lambda _{0}}\text{ strongly in }L^{1}\left( \left(
0,T\right) \times B\left( 0,r\right) \right) \text{, }\forall T>0\text{, }%
\forall r>0\text{,} \\
\frac{1}{\lambda _{k_{m}}^{2}}F\left( \lambda _{k_{m}}v_{m}\right)
\rightarrow \Psi _{\lambda _{0}}\text{ weakly in }L^{\frac{n+1}{n}}\left(
\left( 0,T\right) \times
\mathbb{R}
^{n}\right) \text{. \ \ \ \ \ \ \ \ \ \ \ \ \ \ \ \ \ \ \ \ \ \ \ \ }%
\end{array}%
\right.   \tag{2.23}
\end{equation}%
Then by (1.6), (2.19) and (2.23), we get%
\begin{equation*}
\underset{m\rightarrow \infty }{\lim \sup }\underset{l\rightarrow \infty }{%
\lim \sup }\int\nolimits_{0}^{T}\int\nolimits_{t}^{T}\int\nolimits_{%
\mathbb{R}
^{n}}\left( \frac{1}{\lambda _{k_{l}}}f\left( \lambda _{k_{l}}v_{l}\left(
s,x\right) \right) -\frac{1}{\lambda _{k_{m}}}f\left( \lambda
_{k_{m}}v_{m}\left( s,x\right) \right) \right) \left( v_{mt}\left(
s,x\right) -v_{lt}\left( s,x\right) \right) dxdsdt
\end{equation*}%
\begin{equation*}
=\underset{m\rightarrow \infty }{\lim \sup }\underset{l\rightarrow \infty }{%
\lim \sup }\int\nolimits_{0}^{T}\int\nolimits_{t}^{T}\int\nolimits_{%
\mathbb{R}
^{n}}\left( -\frac{1}{\lambda _{k_{l}}}f\left( \lambda _{k_{l}}v_{l}\left(
s,x\right) \right) v_{lt}\left( s,x\right) \right.
\end{equation*}%
\begin{equation*}
\left. -\frac{1}{\lambda _{k_{m}}}f\left( \lambda _{k_{m}}v_{m}\left(
s,x\right) \right) v_{mt}\left( s,x\right) +2\frac{\partial }{\partial s}%
\Psi _{\lambda _{0}}\left( v(s,x)\right) \right) dxdsdt
\end{equation*}%
\begin{equation*}
\leq T\underset{m\rightarrow \infty }{\lim \sup }\underset{l\rightarrow
\infty }{\lim \sup }\int\nolimits_{%
\mathbb{R}
^{n}}\left( -\frac{1}{\lambda _{k_{l}}^{2}}F\left( \lambda
_{k_{l}}v_{l}\left( T,x\right) \right) -\frac{1}{\lambda _{k_{m}}^{2}}%
F\left( \lambda _{k_{m}}v_{m}\left( T,x\right) \right) +2\Psi _{\lambda
_{0}}\left( v(T,x)\right) \right) dx
\end{equation*}%
\begin{equation*}
+\underset{m\rightarrow \infty }{\lim \sup }\underset{l\rightarrow \infty }{%
\lim \sup }\int\nolimits_{0}^{T}\int\nolimits_{B\left( 0,r\right) }\left(
\frac{1}{\lambda _{k_{l}}^{2}}F\left( \lambda _{k_{l}}v_{l}\left( t,x\right)
\right) +\frac{1}{\lambda _{k_{m}}^{2}}F\left( \lambda _{k_{m}}v_{m}\left(
t,x\right) \right) -2\Psi _{\lambda _{0}}\left( v(t,x)\right) \right) dxdt
\end{equation*}%
\begin{equation*}
+\underset{m\rightarrow \infty }{\lim \sup }\underset{l\rightarrow \infty }{%
\lim \sup }\int\nolimits_{0}^{T}\int\nolimits_{%
\mathbb{R}
^{n}\backslash B\left( 0,r\right) }\left( \frac{1}{\lambda _{k_{l}}^{2}}%
F\left( \lambda _{k_{l}}v_{l}\left( t,x\right) \right) +\frac{1}{\lambda
_{k_{m}}^{2}}F\left( \lambda _{k_{m}}v_{m}\left( t,x\right) \right) -2\Psi
_{\lambda _{0}}\left( v(t,x)\right) \right) dxdt
\end{equation*}%
\begin{equation*}
=-T\underset{m\rightarrow \infty }{\lim \inf }\underset{l\rightarrow \infty }%
{\lim \inf }\int\nolimits_{%
\mathbb{R}
^{n}}\left( \frac{1}{\lambda _{k_{l}}^{2}}F\left( \lambda
_{k_{l}}v_{l}\left( T,x\right) \right) +\frac{1}{\lambda _{k_{m}}^{2}}%
F\left( \lambda _{k_{m}}v_{m}\left( T,x\right) \right) -2\Psi _{\lambda
_{0}}\left( v(T,x)\right) \right) dx
\end{equation*}%
\begin{equation*}
+\underset{m\rightarrow \infty }{\lim \sup }\underset{l\rightarrow \infty }{%
\lim \sup }\int\nolimits_{0}^{T}\int\nolimits_{%
\mathbb{R}
^{n}\backslash B\left( 0,r\right) }\left( \frac{1}{\lambda _{k_{l}}^{2}}%
F\left( \lambda _{k_{l}}v_{l}\left( t,x\right) \right) +\frac{1}{\lambda
_{k_{m}}^{2}}F\left( \lambda _{k_{m}}v_{m}\left( t,x\right) \right) -2\Psi
_{\lambda _{0}}\left( v(t,x)\right) \right) dxdt
\end{equation*}%
\begin{equation*}
\leq \underset{m\rightarrow \infty }{\lim \sup }\underset{l\rightarrow
\infty }{\lim \sup }\int\nolimits_{0}^{T}\int\nolimits_{%
\mathbb{R}
^{n}\backslash B\left( 0,r\right) }\left( \frac{1}{\lambda _{k_{l}}^{2}}%
F\left( \lambda _{k_{l}}v_{l}\left( t,x\right) \right) +\frac{1}{\lambda
_{k_{m}}^{2}}F\left( \lambda _{k_{m}}v_{m}\left( t,x\right) \right) \right)
dxdt
\end{equation*}%
\begin{equation*}
\leq \underset{m\rightarrow \infty }{c_{7}\lim \sup }\int\nolimits_{0}^{T}%
\left( \left\Vert v_{m}\left( t\right) \right\Vert _{L^{2}\left(
\mathbb{R}
^{n}\backslash B\left( 0,r\right) \right) }^{2}+\left\Vert \Delta
v_{m}\left( t\right) \right\Vert _{L^{2}\left(
\mathbb{R}
^{n}\backslash B\left( 0,r\right) \right) }^{2}\right) dt
\end{equation*}%
\begin{equation}
\leq c_{8}\left( 1+\frac{T}{r}\right) \text{, }\forall T\geq 0\text{, }%
\forall r\geq r_{0}\text{,}  \tag{2.24}
\end{equation}%
for $\left\{ \lambda _{k_{m}}\right\} _{m=1}^{\infty }\subset \left( 0,M%
\right] $, where constants $c_{7}$ and $c_{8}$ depend on $M$ and $\underset{m%
}{\sup }\left\Vert v_{m}\right\Vert _{L^{\infty }(0,\infty ;H^{2}(%
\mathbb{R}
^{n}))}$. If, additionally, condition (1.7) is satisfied, then (2.24) holds
for $\left\{ \lambda _{k_{m}}\right\} _{m=1}^{\infty }\subset (0,\infty ),$
with constants $c_{7}$ and $c_{8}$ depending on $\left\Vert f^{\prime
}\right\Vert _{L^{\infty }\left(
\mathbb{R}
\right) }$ and $\underset{m}{\sup }\left\Vert v_{m}\right\Vert _{L^{\infty
}(0,\infty ;H^{2}(%
\mathbb{R}
^{n}))}$. Taking into account (2.24) in (2.22), we get%
\begin{equation*}
\underset{m\rightarrow \infty }{\lim \sup }\underset{l\rightarrow \infty }{%
\lim \sup }\left\Vert S^{\lambda _{k_{m}}}(T+t_{k_{m}}-T_{0})\varphi
_{k_{m}}-S^{\lambda _{k_{l}}}(T+t_{k_{l}}-T_{0})\varphi _{k_{l}}\right\Vert
_{H^{2}\left(
\mathbb{R}
^{n}\right) \times L^{2}\left(
\mathbb{R}
^{n}\right) }^{2}\text{ }
\end{equation*}%
\begin{equation*}
\leq c_{9}(\frac{1}{T}+\frac{1}{r})\text{, }\forall T\geq 1\text{, }\forall
r\geq r_{0}\text{. }
\end{equation*}%
Choosing $T=T_{0}$ in the above inequality, we have%
\begin{equation*}
\underset{m\rightarrow \infty }{\lim \sup }\underset{l\rightarrow \infty }{%
\lim \sup }\left\Vert S^{\lambda _{k_{m}}}(t_{k_{m}})\varphi
_{k_{m}}-S^{\lambda _{k_{l}}}(t_{k_{l}})\varphi _{k_{l}}\right\Vert
_{H^{2}\left(
\mathbb{R}
^{n}\right) \times L^{2}\left(
\mathbb{R}
^{n}\right) }^{2}\leq c_{9}(\frac{1}{T_{0}}+\frac{1}{r})\text{, }\forall
T_{0}\geq 1\text{, }\forall r\geq r_{0}\text{,}
\end{equation*}%
and consequently%
\begin{equation*}
\underset{k\rightarrow \infty }{\lim \inf }\underset{m\rightarrow \infty }{%
\lim \inf }\left\Vert S^{\lambda _{k}}(t_{k})\varphi _{k}-S^{\lambda
_{m}}(t_{m})\varphi _{m}\right\Vert _{H^{2}\left(
\mathbb{R}
^{n}\right) \times L^{2}\left(
\mathbb{R}
^{n}\right) }=0\text{.}
\end{equation*}%
Thus, by using the argument at the end of the proof of \cite[Lemma 3.4]{19},
we complete the proof of the lemma.
\end{proof}

\begin{lemma}
Assume that conditions (1.3)-(1.6) hold and $\lambda \in \lbrack 0,\infty ]$%
. Then%
\begin{equation*}
\underset{t\rightarrow \infty }{\lim }\underset{\left( u_{0},u_{1}\right)
\in B}{\sup }\left\Vert S^{\lambda }\left( t\right) \left(
u_{0},u_{1}\right) \right\Vert _{H^{2}\left(
\mathbb{R}
^{n}\right) \times L^{2}\left(
\mathbb{R}
^{n}\right) }=0\text{,}
\end{equation*}%
for every bounded subset $B\subset H^{2}\left(
\mathbb{R}
^{n}\right) \times L^{2}\left(
\mathbb{R}
^{n}\right) $.
\end{lemma}

\begin{proof}
Let $B\subset H^{2}\left(
\mathbb{R}
^{n}\right) \times L^{2}\left(
\mathbb{R}
^{n}\right) $. From Lemma 2.3, it follows that the $\omega $-limit set of $B$%
, namely%
\begin{equation*}
\omega _{\lambda }\left( B\right) =\underset{\tau \geq 0}{\cap }\overline{%
\underset{t\geq \tau }{\cup }S^{\lambda }\left( t\right) B}
\end{equation*}%
is compact in $H^{2}\left(
\mathbb{R}
^{n}\right) \times L^{2}\left(
\mathbb{R}
^{n}\right) $, invariant with respect to $S^{\lambda }\left( t\right) $ and%
\begin{equation*}
\underset{t\rightarrow \infty }{\lim }\underset{\varphi \in B}{\sup }\text{ }%
\underset{\psi \in \omega _{\lambda }\left( B\right) }{\inf }\left\Vert
S^{\lambda }\left( t\right) \varphi -\psi \right\Vert _{H^{2}\left(
\mathbb{R}
^{n}\right) \times L^{2}\left(
\mathbb{R}
^{n}\right) }=0\text{.}
\end{equation*}%
Our aim is to show that $\omega _{\lambda }\left( B\right) \equiv \{(0,0)\}$%
. Since $\omega _{\lambda }\left( B\right) $ is invariant, it is enough to
show that%
\begin{equation}
\underset{t\rightarrow \infty }{\lim }\underset{\left( u_{0},u_{1}\right)
\in \omega _{\lambda }\left( B\right) }{\sup }\left\Vert S^{\lambda }\left(
t\right) \left( u_{0},u_{1}\right) \right\Vert _{H^{2}\left(
\mathbb{R}
^{n}\right) \times L^{2}\left(
\mathbb{R}
^{n}\right) }=0\text{.}  \tag{2.25}
\end{equation}%
Let $\left( u_{0},u_{1}\right) \in \omega _{\lambda }\left( B\right) $ and $%
\left( u_{\lambda }\left( t\right) ,u_{\lambda t}\left( t\right) \right)
=S^{\lambda }\left( t\right) \left( u_{0},u_{1}\right) $. Multiplying (2.9)$%
_{1}$ by $u_{\lambda t}$ and integrating over $\left( s,t\right) \times
\mathbb{R}
^{n}$, for the energy functional%
\begin{equation*}
E_{\lambda }(t,u_{0},u_{1})=\frac{1}{2}\left\Vert u_{\lambda t}\left(
t\right) \right\Vert _{L^{2}\left(
\mathbb{R}
^{n}\right) }^{2}
\end{equation*}%
\begin{equation*}
+\frac{1}{2}\left\Vert \Delta u_{\lambda }\left( t\right) \right\Vert
_{L^{2}\left(
\mathbb{R}
^{n}\right) }^{2}+\frac{\alpha }{2}\left\Vert u_{\lambda }\left( t\right)
\right\Vert _{L^{2}\left(
\mathbb{R}
^{n}\right) }^{2}+\int\nolimits_{%
\mathbb{R}
^{n}}\Psi _{\lambda }\left( u_{\lambda }\left( t,x\right) \right) dx\text{,}
\end{equation*}%
we have
\begin{equation}
E_{\lambda }(t,u_{0},u_{1})+\int\nolimits_{s}^{t}\int\nolimits_{%
\mathbb{R}
^{n}}a\left( x\right) \left\vert u_{\lambda t}\left( t,x\right) \right\vert
^{2}dxdt=E_{\lambda }(s,u_{0},u_{1})\text{, }\forall t\geq s\text{. }
\tag{2.26}
\end{equation}%
So, $E_{\lambda }(t,u_{0},u_{1})$ is nonincreasing with respect to $t$. To
prove (2.25), it is enough to show that
\begin{equation}
\underset{t\rightarrow \infty }{\lim }\underset{\left( u_{0},u_{1}\right)
\in \omega _{\lambda }\left( B\right) }{\sup }E_{\lambda }(t,u_{0},u_{1})=0%
\text{.}  \tag{2.27}
\end{equation}%
Assume that (2.27) is not true. Then there exist $\epsilon >0$, $%
t_{k}\rightarrow \infty $ and the sequence $\left\{ \left(
u_{0k},u_{1k}\right) \right\} _{k=1}^{\infty }\subset \omega _{\lambda
}\left( B\right) $ such that%
\begin{equation}
E_{\lambda }(t_{k},u_{0k},u_{1k})\geq \epsilon \text{.}  \tag{2.28}
\end{equation}%
Since $\omega _{\lambda }\left( B\right) $ is compact, the sequence $\left\{
\left( u_{0k},u_{1k}\right) \right\} _{k=1}^{\infty }$ has a convergent
subsequence with limit in $\omega _{\lambda }\left( B\right) $. Without loss
of generality, denote this subsequence again by $\left\{ \left(
u_{0k},u_{1k}\right) \right\} _{k=1}^{\infty }$. Then we have%
\begin{equation*}
\left( u_{0k},u_{1k}\right) \rightarrow \left( v,w\right) \text{ strongly in
}H^{2}\left(
\mathbb{R}
^{n}\right) \times L^{2}\left(
\mathbb{R}
^{n}\right) \text{,}
\end{equation*}%
where $\left( v,w\right) \in \omega _{\lambda }\left( B\right) $. Hence,
since $E_{\lambda }(t,\cdot ,\cdot )$ is continuous functional on $%
H^{2}\left(
\mathbb{R}
^{n}\right) \times L^{2}\left(
\mathbb{R}
^{n}\right) $, by Lemma 2.2, it follows that%
\begin{equation}
\underset{k\rightarrow \infty }{\lim }E_{\lambda
}(t,u_{0k},u_{1k})=E_{\lambda }(t,v,w)\text{, }\forall t\geq 0\text{.}
\tag{2.29}
\end{equation}%
Since the stationary equation corresponding to (2.9)$_{1}$ has only zero
solution, applying \cite[Theorem 2]{18}, we find that%
\begin{equation*}
\underset{t\rightarrow \infty }{\lim }\left\Vert S^{\lambda }\left( t\right)
x\right\Vert _{H^{2}\left(
\mathbb{R}
^{n}\right) \times L^{2}\left(
\mathbb{R}
^{n}\right) }=0\text{, \ \ }\forall x\in H^{2}\left(
\mathbb{R}
^{n}\right) \times L^{2}\left(
\mathbb{R}
^{n}\right) \text{.}
\end{equation*}%
Then for any $\epsilon >0$ there exists $t_{\epsilon }^{\prime }$ such that%
\begin{equation*}
E_{\lambda }(t_{\epsilon }^{\prime },v,w)<\frac{\epsilon }{2}\text{, }
\end{equation*}%
which, together with (2.29), yields that%
\begin{equation*}
E_{\lambda }(t_{\epsilon }^{\prime },u_{0k},u_{1k})<\epsilon \text{, }
\end{equation*}%
for large enough $k$. Since $E_{\lambda }(t,u_{0},u_{1})$ is nonincreasing
with respect to $t$, the last inequality contradicts (2.28). So, our
assumption is false, i.e. (2.27) is true and proof is completed.
\end{proof}

\begin{lemma}
Assume conditions (1.3)-(1.6) hold. Then%
\begin{equation*}
\underset{t\rightarrow \infty }{\lim }\underset{\lambda \in \left( 0,M\right]
}{\sup }\underset{\left( u_{0},u_{1}\right) \in B}{\sup }\left\Vert
S^{\lambda }\left( t\right) \left( u_{0},u_{1}\right) \right\Vert
_{H^{2}\left(
\mathbb{R}
^{n}\right) \times L^{2}\left(
\mathbb{R}
^{n}\right) }=0\text{,}
\end{equation*}%
for every bounded subset $B\subset H^{2}\left(
\mathbb{R}
^{n}\right) \times L^{2}\left(
\mathbb{R}
^{n}\right) $ and $M>0$.
\end{lemma}

\begin{proof}
Taking into account (2.12), for any $\lambda \in \left( 0,M\right] $\ and $%
\left( u_{0},u_{1}\right) \in B$,\ we have%
\begin{equation}
\left\Vert S^{\lambda }\left( t\right) \left( u_{0},u_{1}\right) \right\Vert
_{H^{2}\left(
\mathbb{R}
^{n}\right) \times L^{2}\left(
\mathbb{R}
^{n}\right) }\leq r\text{,}  \tag{2.30}
\end{equation}%
where $r$ depends on $M$ and $B$, and is independent of $\lambda \in \left(
0,M\right] $, $t$ and $\left( u_{0},u_{1}\right) $. We will prove Lemma 2.5
by contradiction. Assume that Lemma 2.5 is not true. Then there exist $%
\epsilon >0$, sequences $\left\{ \lambda _{k}\right\} _{k=1}^{\infty
}\subset \left( 0,M\right] $, $\left\{ \left( u_{0k},u_{1k}\right) \right\}
_{k=1}^{\infty }\subset B$ and $t_{k}\rightarrow \infty $ such that%
\begin{equation}
\left\Vert S^{\lambda _{k}}\left( t_{k}\right) \left( u_{0k},u_{1k}\right)
\right\Vert _{H^{2}\left(
\mathbb{R}
^{n}\right) \times L^{2}\left(
\mathbb{R}
^{n}\right) }\geq \epsilon \text{.}  \tag{2.31}
\end{equation}%
For any $t_{k}\geq t$ consider the sequence $\left\{ S^{\lambda _{k}}\left(
t\right) S^{\lambda _{k}}\left( t_{k}-t\right) \left( u_{0k},u_{1k}\right)
\right\} _{k=1}^{\infty }$. By Lemma 2.3, the sequence $\left\{ S^{\lambda
_{k}}\left( t_{k}-t\right) \left( u_{0k},u_{1k}\right) \right\}
_{k=1}^{\infty }$ is relatively compact in $H^{2}\left(
\mathbb{R}
^{n}\right) \times L^{2}\left(
\mathbb{R}
^{n}\right) $. Then it has a convergent subsequence $\left\{ S^{\lambda
_{k_{m}}}\left( t_{k_{m}}-t\right) \left( u_{0k_{m}},u_{1k_{m}}\right)
\right\} _{m=1}^{\infty }$ with limit $\varphi _{0}\in H^{2}\left(
\mathbb{R}
^{n}\right) \times L^{2}\left(
\mathbb{R}
^{n}\right) $ and by Lemma 2.2,%
\begin{equation}
S^{\lambda _{k_{m}}}\left( t\right) S^{\lambda _{k_{m}}}\left(
t_{k_{m}}-t\right) \left( u_{0k_{m}},u_{1k_{m}}\right) \rightarrow
S^{\lambda _{0}}\left( t\right) \varphi _{0}\text{ strongly in }H^{2}\left(
\mathbb{R}
^{n}\right) \times L^{2}\left(
\mathbb{R}
^{n}\right) \text{,}  \tag{2.32}
\end{equation}%
where $\lambda _{0}\in \left[ 0,M\right] $ is the limit of $\left\{ \lambda
_{k_{m}}\right\} _{m=1}^{\infty }$. Furthermore, from (2.30) it follows that
\newline
$\left\{ S^{\lambda _{k_{m}}}\left( t_{k_{m}}-t\right) \left(
u_{0k_{m}},u_{1k_{m}}\right) \right\} _{m=1}^{\infty }\subset \mathcal{B}%
\left( 0,r\right) $, where $\mathcal{B}\left( 0,r\right) $ $=\left\{ \varphi
\in H^{2}\left(
\mathbb{R}
^{n}\right) \times L^{2}\left(
\mathbb{R}
^{n}\right) :\right. $\newline
$\left. \left\Vert \varphi \right\Vert _{H^{2}\left(
\mathbb{R}
^{n}\right) \times L^{2}\left(
\mathbb{R}
^{n}\right) }\leq r\right\} $. Consequently, $\varphi _{0}\in $ $\mathcal{B}%
\left( 0,r\right) $ and by previous lemma, for any $\epsilon >0$, there
exists $t_{\epsilon }$ such that%
\begin{equation*}
\underset{\varphi \in \mathcal{B}\left( 0,r\right) }{\sup }\left\Vert
S^{\lambda _{0}}\left( t\right) \varphi \right\Vert _{H^{2}\left(
\mathbb{R}
^{n}\right) \times L^{2}\left(
\mathbb{R}
^{n}\right) }<\frac{\epsilon }{2}\text{,}
\end{equation*}%
for $t\geq t_{\epsilon }$. Taking into account (2.32) for $t_{k_{m}}\geq
t_{\epsilon }$ and choosing $t=t_{\epsilon }$, we get%
\begin{equation*}
\left\Vert S^{\lambda _{k_{m}}}\left( t_{k_{m}}\right) \left(
u_{0k_{m}},u_{1k_{m}}\right) \right\Vert _{H^{2}\left(
\mathbb{R}
^{n}\right) \times L^{2}\left(
\mathbb{R}
^{n}\right) }<\epsilon \text{,}
\end{equation*}%
for large enough $k$, which contradicts (2.31). So, our assumption is false
and proof is completed.
\end{proof}

\begin{lemma}
Assume that conditions of Theorem 1.1 hold. Then there exist $t_{0}>0$ and $%
C\in \left( 0,1\right) $ such that the estimate
\begin{equation}
E_{\lambda }\left( t_{0},u_{0},u_{1}\right) \leq C  \tag{2.33}
\end{equation}%
holds for all $\lambda >0$ and $\left( u_{0},u_{1}\right) \in H^{2}\left(
\mathbb{R}
^{n}\right) \times L^{2}\left(
\mathbb{R}
^{n}\right) $ satisfying the condition $E_{\lambda }\left(
0,u_{0},u_{1}\right) =1$.
\end{lemma}

\begin{proof}
We will prove lemma by contradiction. Assume that (2.33) is not true. Then
there exist sequences $C_{k}\nearrow 1$, $t_{k}\rightarrow \infty $, $%
\left\{ \lambda _{k}\right\} _{k=1}^{\infty }\subset \left( 0,\infty \right)
$ and $\left\{ \left( u_{0k},u_{1k}\right) \right\} _{k=1}^{\infty }$ $%
\subset H^{2}\left(
\mathbb{R}
^{n}\right) \times L^{2}\left(
\mathbb{R}
^{n}\right) $ satisfying the condition $E_{\lambda _{k}}\left(
0,u_{0k},u_{1k}\right) =1$ such that%
\begin{equation}
E_{\lambda _{k}}\left( t_{k},u_{0k},u_{1k}\right) >C_{k}\text{.}  \tag{2.34}
\end{equation}%
Assume that $\left\{ \lambda _{k}\right\} _{k=1}^{\infty }\subset \left( 0,M%
\right] $, for some $M>0$. Then from Lemma 2.5, we get%
\begin{equation}
\underset{t_{k}\rightarrow \infty }{\lim }\underset{\left(
u_{0},u_{1}\right) \in B_{0}}{\sup }\left\Vert S^{\lambda _{k}}\left(
t_{k}\right) \left( u_{0},u_{1}\right) \right\Vert _{H^{2}\left(
\mathbb{R}
^{n}\right) \times L^{2}\left(
\mathbb{R}
^{n}\right) }=0\text{,}  \tag{2.35}
\end{equation}%
where $B_{0}=\cup _{k=1}^{\infty }\left\{ \left( u_{0},u_{1}\right) \in
H^{2}\left(
\mathbb{R}
^{n}\right) \times L^{2}\left(
\mathbb{R}
^{n}\right) :E_{\lambda _{k}}\left( 0,u_{0},u_{1}\right) \leq 1\right\} $.
On the other hand, by (1.3), (1.6) and (2.11), we get%
\begin{equation*}
E_{\lambda _{k}}\left( t_{k},u_{0k},u_{1k}\right) \leq \widetilde{c}\left(
M\right) \left( \left\Vert u_{kt}\left( t\right) \right\Vert _{L^{2}(%
\mathbb{R}
^{n})}^{2}+\left\Vert u_{k}\left( t\right) \right\Vert _{H^{2}(%
\mathbb{R}
^{n})}^{2}+\left\Vert u_{k}\left( t\right) \right\Vert _{H^{2}(%
\mathbb{R}
^{n})}^{p+1}\right) \text{,}
\end{equation*}%
where $\left( u_{k}\left( t\right) ,u_{kt}\left( t\right) \right)
=S^{\lambda _{k}}\left( t\right) \left( u_{0k},u_{1k}\right) $. Hence, from
(2.35), we have%
\begin{equation*}
\underset{t_{k}\rightarrow \infty }{\lim }\underset{\left(
u_{0},u_{1}\right) \in B_{0}}{\sup }E_{\lambda _{k}}\left(
t_{k},u_{0k},u_{1k}\right) =0\text{,}
\end{equation*}%
which contradicts (2.34). So, the sequence $\left\{ \lambda _{k}\right\}
_{k=1}^{\infty }$ must have a subsequence which goes to infinity. Without
loss of generality, assume that $\lambda _{k}\rightarrow \infty $. Now, we
consider the globally Lipschitz case and the superlinear case separately. \

$\left( i\right) $\textit{The globally Lipschitz case}:\ \ Since the
nonlinear function $f$ is globally Lipschitz and \newline
$E_{\lambda _{k}}\left( 0,u_{0k},u_{1k}\right) =1$, we obtain%
\begin{equation}
E_{\lambda _{k}}\left( t_{k},u_{0k},u_{1k}\right) \leq c_{0}\left\Vert
S^{\lambda _{k}}\left( t_{k},u_{0k},u_{1k}\right) \right\Vert _{H^{2}\left(
\mathbb{R}
^{n}\right) \times L^{2}\left(
\mathbb{R}
^{n}\right) }\text{.}  \tag{2.36}
\end{equation}%
For any $t_{k}\geq t$ consider the sequence $\left\{ S^{\lambda _{k}}\left(
t\right) S^{\lambda _{k}}\left( t_{k}-t\right) \left( u_{0k},u_{1k}\right)
\right\} _{k=1}^{\infty }$. By Lemma 2.3, the sequence $\left\{ S^{\lambda
_{k}}\left( t_{k}-t\right) \left( u_{0k},u_{1k}\right) \right\}
_{k=1}^{\infty }$ is relatively compact in $H^{2}\left(
\mathbb{R}
^{n}\right) \times L^{2}\left(
\mathbb{R}
^{n}\right) $ if $t_{k}\rightarrow \infty $. Then it has a convergent
subsequence $\left\{ S^{\lambda _{k_{m}}}\left( t_{k_{m}}-t\right) \left(
u_{0k_{m}},u_{1k_{m}}\right) \right\} _{m=1}^{\infty }$ with the limit $%
\varphi _{0}$ $\epsilon $ $\overline{B_{0}}$ and by Lemma 2.2,%
\begin{equation}
S^{\lambda _{k_{m}}}\left( t\right) S^{\lambda _{k_{m}}}\left(
t_{k_{m}}-t\right) \left( u_{0k_{m}},u_{1k_{m}}\right) \rightarrow
S^{\lambda _{0}}\left( t\right) \varphi _{0}\text{ strongly in }H^{2}\left(
\mathbb{R}
^{n}\right) \times L^{2}\left(
\mathbb{R}
^{n}\right) \text{,}  \tag{2.37}
\end{equation}%
where $\lambda _{0}=\infty $. On the other hand, by Lemma 2.4,%
\begin{equation*}
\underset{t\rightarrow \infty }{\lim }\underset{\left( u_{0},u_{1}\right)
\in \overline{B_{0}}}{\sup }\left\Vert S^{\lambda _{0}}\left( t\right)
\left( u_{0},u_{1}\right) \right\Vert _{H^{2}\left(
\mathbb{R}
^{n}\right) \times L^{2}\left(
\mathbb{R}
^{n}\right) }=0
\end{equation*}%
and then for any $\epsilon >0$ there exists $t_{\epsilon }$ such that%
\begin{equation*}
\underset{\left( u_{0},u_{1}\right) \in \overline{B_{0}}}{\sup }\left\Vert
S^{\lambda _{0}}\left( t\right) \left( u_{0},u_{1}\right) \right\Vert
_{H^{2}\left(
\mathbb{R}
^{n}\right) \times L^{2}\left(
\mathbb{R}
^{n}\right) }<\frac{\epsilon }{2}\text{, }\forall t\geq t_{\epsilon }\text{.}
\end{equation*}%
Choosing $t=t_{\epsilon }$ in (2.37), we get%
\begin{equation*}
\left\Vert S^{\lambda _{k_{m}}}\left( t_{k_{m}}\right) \left(
u_{0k_{m}},u_{1k_{m}}\right) \right\Vert _{H^{2}\left(
\mathbb{R}
^{n}\right) \times L^{2}\left(
\mathbb{R}
^{n}\right) }<\epsilon \text{, }
\end{equation*}%
for large enough $m$, which, together with (2.36), contradicts (2.34). So,
our assumption is false and proof is completed for the globally Lipschitz
case.

$\left( ii\right) $ \textit{The superlinear case}: As mentioned in \cite[%
Remark 3.2]{3}, using techniques of that article, one can show that there
exists some $T_{1}>0$ such that for every $T>T_{1}$ there exists a constant $%
C\left( T\right) >0$ so that the following estimate holds%
\begin{equation}
E_{\lambda _{k}}\left( T,u_{0k},u_{1k}\right) \leq C\left( T\right) \left(
\int\nolimits_{0}^{T}\int\nolimits_{%
\mathbb{R}
^{n}}a\left( x\right) \left\vert u_{\lambda _{k}t}\left( t,x\right)
\right\vert ^{2}dxdt+\left\Vert u_{\lambda _{k}}\right\Vert _{L^{2}\left(
\left( 0,T\right) \times B(0,4r_{0})\right) }^{2}\right) \text{,}  \tag{2.38}
\end{equation}%
where $\left( u_{\lambda _{k}}\left( t\right) ,u_{\lambda _{k}t}\left(
t\right) \right) =S^{\lambda _{k}}\left( t\right) \left(
u_{0k},u_{1k}\right) $. The constant $C\left( T\right) $ only depends on the
nonlinearity $f$ and the constant $\delta $ in superlinear case (see \cite{3}%
, for details).\newline
\ \ Taking into account $E_{\lambda _{k}}\left( 0,u_{0k},u_{1k}\right) =1$
and (2.11), we have%
\begin{equation}
\int\nolimits_{0}^{T}\int\nolimits_{%
\mathbb{R}
^{n}}a\left( x\right) \left\vert u_{\lambda _{k}t}\left( t,x\right)
\right\vert ^{2}dxdt+\left\Vert u_{\lambda _{k}}\right\Vert _{L^{2}\left(
\left( 0,T\right) \times B(0,4r_{0})\right) }^{2}\leq \widehat{C}_{1}\text{.}
\tag{2.39}
\end{equation}%
By (2.38) and (2.39), it follows that the sequence $\left\{ F_{\lambda
_{k}}\left( u_{\lambda _{k}}\right) \right\} _{k=1}^{\infty }$ is bounded in
$L^{1}\left( \left( 0,T\right) \times B(0,4r_{0})\right) $, where%
\begin{equation}
F_{\lambda }\left( z\right) =\frac{1}{\lambda }\int\nolimits_{0}^{z}f\left(
\lambda s\right) ds=\frac{1}{\lambda ^{2}}F\left( \lambda z\right) \text{, \
\ }\forall \lambda >0\text{.}  \tag{2.40}
\end{equation}%
On the other hand, the condition (1.8) implies%
\begin{equation}
F\left( s\right) \geq c\left\vert s\right\vert ^{2+\delta }\text{, }\forall
\left\vert s\right\vert \geq 1  \tag{2.41}
\end{equation}%
with $c=\min \left\{ F\left( 1\right) ,F\left( -1\right) \right\} $.
Combining (2.38)-(2.41), we get%
\begin{equation*}
\lambda _{k}^{\delta }\int \int\nolimits_{\left\{ \left\vert u_{\lambda
_{k}}\right\vert \geq \lambda _{k}^{-1}\right\} \cap \left\{ \left(
0,T\right) \times B(0,4r_{0})\right\} }\left\vert u_{\lambda
_{k}}\right\vert ^{2+\delta }dxdt\leq \widehat{C}_{2}\text{,}
\end{equation*}%
which implies%
\begin{equation}
\underset{k\rightarrow \infty }{\lim }\int\nolimits_{0}^{T}\int%
\nolimits_{B(0,4r_{0})}\left\vert u_{\lambda _{k}}\right\vert ^{2+\delta
}dxdt=0\text{.}  \tag{2.42}
\end{equation}%
Moreover, by (2.26), we have%
\begin{equation*}
0\leq \int\nolimits_{0}^{t_{k}}\int\nolimits_{%
\mathbb{R}
^{n}}a\left( x\right) \left\vert u_{\lambda _{k}t}\left( t,x\right)
\right\vert ^{2}dxdt=1-E_{\lambda _{k}}\left( t_{k},u_{0k},u_{1k}\right)
\leq 1-C_{k}
\end{equation*}%
and since $t_{k}\rightarrow \infty $, $C_{k}\nearrow 1$, we get%
\begin{equation}
\underset{k\rightarrow \infty }{\lim }\int\nolimits_{0}^{T}\int\nolimits_{%
\mathbb{R}
^{n}}a\left( x\right) \left\vert u_{\lambda _{k}t}\left( t,x\right)
\right\vert ^{2}dxdt=0\text{,}  \tag{2.43}
\end{equation}%
for every $T>0$. Taking into account (2.42) and (2.43) in (2.38), we deduce%
\begin{equation*}
\underset{k\rightarrow \infty }{\lim }E_{\lambda _{k}}\left(
T,u_{0k},u_{1k}\right) =0
\end{equation*}%
and since the energy functional $E_{\lambda }\left( t,u_{0},u_{1}\right) $
is nonincreasing with respect to $t$, we obtain%
\begin{equation*}
\underset{k\rightarrow \infty }{\lim }E_{\lambda _{k}}\left(
t_{k},u_{0k},u_{1k}\right) =0
\end{equation*}%
which contradicts (2.34). Hence our assumption is false and the proof is
completed for superlinear case.
\end{proof}

Now we can prove the main result. Assume that $u\in C\left( [0,\infty
);H^{2}\left(
\mathbb{R}
^{n}\right) \right) \cap C^{1}\left( [0,\infty );L^{2}\left(
\mathbb{R}
^{n}\right) \right) $ is the solution of problem (1.1)-(1.2) with initial
data $\left( u_{0},u_{1}\right) \in H^{2}\left(
\mathbb{R}
^{n}\right) \times L^{2}\left(
\mathbb{R}
^{n}\right) $ and consider the problem (2.9) with $\lambda =\sqrt{E\left(
0,u_{0},u_{1}\right) }>0$. Then it is easy to see that $u_{\lambda }=\frac{u%
}{\lambda }$ is the solution of problem (2.9) with the initial data $\left(
u_{0\lambda },u_{1\lambda }\right) =\left( \frac{u_{0}}{\lambda },\frac{u_{1}%
}{\lambda }\right) \in $ $H^{2}\left(
\mathbb{R}
^{n}\right) \times L^{2}\left(
\mathbb{R}
^{n}\right) $ and%
\begin{equation*}
E_{\lambda }\left( t,u_{0\lambda },u_{1\lambda }\right) =\frac{1}{2}%
\left\Vert u_{\lambda t}\left( t\right) \right\Vert _{L^{2}(%
\mathbb{R}
^{n})}^{2}+\frac{1}{2}\left\Vert \Delta u_{\lambda }\left( t\right)
\right\Vert _{L^{2}\left(
\mathbb{R}
^{n}\right) }^{2}+\frac{\alpha }{2}\left\Vert u_{\lambda }\left( t\right)
\right\Vert _{L^{2}\left(
\mathbb{R}
^{n}\right) }^{2}
\end{equation*}%
\begin{equation*}
+\int\nolimits_{%
\mathbb{R}
^{n}}F_{\lambda }\left( u_{\lambda }\left( t,x\right) \right) dx=\frac{1}{%
2\lambda ^{2}}\left\Vert u_{t}\left( t\right) \right\Vert _{L^{2}\left(
\mathbb{R}
^{n}\right) }^{2}+\frac{1}{2\lambda ^{2}}\left\Vert \Delta u\left( t\right)
\right\Vert _{L^{2}\left(
\mathbb{R}
^{n}\right) }^{2}+\frac{\alpha }{2\lambda ^{2}}\left\Vert u\left( t\right)
\right\Vert _{L^{2}\left(
\mathbb{R}
^{n}\right) }^{2}
\end{equation*}%
\begin{equation*}
+\frac{1}{\lambda ^{2}}\int\nolimits_{%
\mathbb{R}
^{n}}F\left( u\left( t,x\right) \right) dx=\frac{1}{\lambda ^{2}}E\left(
t,u_{0},u_{1}\right) \text{.}
\end{equation*}%
Then, since $E_{\lambda }\left( 0,u_{0\lambda },u_{1\lambda }\right) =\frac{1%
}{\lambda ^{2}}E\left( 0,u_{0},u_{1}\right) =1$, by Lemma 2.6, there exist $%
t_{0}>0$ and some constant $\beta \in \left( 0,1\right) $ such that
\begin{equation*}
E_{\lambda }\left( t_{0},u_{0\lambda },u_{1\lambda }\right) \leq \beta \text{%
.}
\end{equation*}%
Hence, we have%
\begin{equation*}
E\left( t_{0},u_{0},u_{1}\right) \leq \lambda ^{2}\beta =E\left(
0,u_{0},u_{1}\right) \beta
\end{equation*}%
and by the successive iteration, we obtain%
\begin{equation*}
E\left( nt_{0},u_{0},u_{1}\right) \leq \beta ^{n}E\left(
0,u_{0,}u_{1}\right) \text{, }\forall n\in
\mathbb{N}
\text{,}
\end{equation*}%
for every $\left( u_{0,}u_{1}\right) \in H^{2}\left(
\mathbb{R}
^{n}\right) \times L^{2}\left(
\mathbb{R}
^{n}\right) $. Since the energy functional $E\left( t,u_{0},u_{1}\right) $
is nonincreasing with respect to $t$, for $t=nt_{0}+r$, \ $0\leq r<t_{0}$,\
we find%
\begin{equation*}
E\left( t,u_{0},u_{1}\right) =E\left( nt_{0}+r,u_{0},u_{1}\right) \leq
E\left( nt_{0},u_{0},u_{1}\right) \leq \beta ^{n}E\left(
0,u_{0},u_{1}\right) \text{.}
\end{equation*}%
Now, denoting $\gamma =\frac{1}{t_{0}}\ln (\frac{1}{\beta })$, from the last
inequality, we get%
\begin{equation*}
E\left( t,u_{0},u_{1}\right) \leq e^{-\gamma t_{0}n}E\left(
0,u_{0},u_{1}\right) =e^{-\gamma t}e^{\gamma r}E\left( 0,u_{0},u_{1}\right)
\end{equation*}%
\begin{equation*}
\leq e^{-\gamma t}e^{\gamma t_{0}}E\left( 0,u_{0},u_{1}\right) =CE\left(
0,u_{0},u_{1}\right) e^{-\gamma t}\text{,}
\end{equation*}%
where $C=$ $e^{\gamma t_{0}}$. Hence, the proof of Theorem 1.1 is completed.



\begin{thebibliography}{99}
\bibitem{1} E. Zuazua, Stability and decay for a class of nonlinear
hyperbolic problems, Asymptotic Analysis, 1 (1988) 161--185.

\bibitem{2} E. Zuazua, Exponential decay for the semilinear wave equation
with locally ditributed damping, Comm. Partial Differential Equations, 15
(1990) 205--235.

\bibitem{3} E. Zuazua, Exponential decay for the semilinear wave equation
with localized damping in unbounded domains, J.Math Pures Appl., 70 (1991)
513--529.

\bibitem{4} M. Nakao, Decay of solutions of wave equation with a local
degenerate dissipation, Israel J. Math., 95 (1996) 25--42.

\bibitem{5} M. Nakao, Decay of solutions of wave equation with a local
nonlinear dissipation, Math. Ann., 305 (1996) 403--417.

\bibitem{6} L. Tebou, Stabilization of the wave equation with localized
nonlinear damping, J. Differential Equtions, 145 (1998) 502--524.

\bibitem{7} L. Tebou, Well posedness and energy decay estimates for the
damped wave equations with $L^{r}$ localizing coefficient, Comm.Partial
Differential Equations, 23 (1998) 1839--1855.

\bibitem{8} R. B. Guzman, M. Tucsnak, Energy decay estimates for the damped
plate equation with a local degenerated dissipation, Systems \& Control
Letters, 48 (2003) 191 -- 197.

\bibitem{9} M.M. Cavalcanti, V. N. Domingos Cavalcanti, T. F. Ma,
Exponantial decay of the viscoelastic Euler-Bernoulli equation with a
nonlocal dissipation in general domains, Differential and Integral Equations
17 (2004) 495--510.

\bibitem{10} M.M. Cavalcanti, V. N. Domingos Cavalcanti, J. A. Soriano,
Global existence and asymptotic stability for the nonlinear and generalized
damped extensible plate equation, Communications in Contemporary
Mathematics, 5 (2004) 705-731.

\bibitem{11} L. Tebou, Well-posedness and stability of a hinged plate
equation with a localized nonlinear structural damping, Nonlinear Analysis
71 (2009) 2288--2297.

\bibitem{12} J. Li, Y. Wu, Exponential stability of the plate equations with
potential of second order and indefinite damping, J. Math. Anal. Appl. 359
(2009) 62--75.

\bibitem{13} J.Y. Park, J.R. Kang, Energy decay estimates for the
Bernoulli-Euler type equation with a local degenerate dissipation, Applied
Mathematics Letters, 23 (2010) 1274--1279.

\bibitem{14} A. Ruiz, Unique continuation for weak solutions of the wave
equation plus a potential, J.Math Pures Appl., 710 (1992) 455--467.

\bibitem{15} A. Kh. Khanmamedov, Global attractors for von Karman equations
with nonlinear interior dissipation, J. Math. Anal. Appl. 318 (2006) 92--101.

\bibitem{16} I. Chueshov, I. Lasiecka, Long-time behavior of second order
evolution equations with nonlinear damping, Memoirs of AMS, 195 (2008).

\bibitem{17} I. Chueshov, I. Lasiecka, Von Karman Evolution
Equations:Well-posedness and long-time dynamics, Springer, 2010.

\bibitem{18} A. Kh. Khanmamedov, Global attractors for the plate equation
with localized damping and a critical exponent in an unbounded domain,
J.Differential Equations, 225 (2006) 528--548.

\bibitem{19} A. Kh. Khanmamedov, Global attractors for 2-D wave equations
with displacement dependent damping, Math. Methods Appl. Sci., 33 (2010)
177-187.
\end{thebibliography}
\end{document}